\documentclass[preprint]{amsart}

\usepackage{amsmath,amsthm,amssymb}
\usepackage{bbm}
\usepackage{comment}
\usepackage{graphicx}

\newtheorem{thm}{Theorem}
\newtheorem{prp}{Proposition}
\newtheorem{crl}{Corollary}

\begin{document}
\title{Information geometry of L\'evy processes and financial models}
\author{Jaehyung Choi}
\address{}
\email{jj.jaehyung.choi@gmail.com}

\begin{abstract}
	We develop the information geometry of L\'evy processes. The $\alpha$-divergence is derived directly in terms of the L\'evy triplets, and the corresponding Fisher information matrix and $\alpha$-connection are obtained from it. Potential functions and dually flat structures for a class of parametric L\'evy processes are also investigated. In addition, we discuss statistical implications of this information geometry, including bias reduction and Bayesian predictive priors. Several L\'evy processes widely used in financial modeling, such as tempered stable processes, the CGMY model, variance gamma processes, and the Merton model, are investigated as illustrative examples.
\end{abstract}

\maketitle
\section{Introduction}
\label{sec_intro}
	Information geometry studies the differential-geometric structure of statistical models by representing families of probability distributions as Riemannian manifolds. This representation provides geometric interpretations of statistical inference and parameter estimation procedures \cite{efron1975defining, efron1978geometry,firth1993bias,amari2000methods,komaki2006shrinkage,kass2011geometrical}. Geometric structures of statistical manifolds, such as metric tensors and connections, are derived from divergence functions. In particular, the $\alpha$-divergence generalizes the Kullback--Leibler divergence and generates the Fisher information matrix and the $\alpha$-connection. Information geometry has been applied to exponential families \cite{efron1975defining, efron1978geometry,firth1993bias,kosmidis2009bias,kosmidis2010generic,amari2000methods,kass2011geometrical}, time series models \cite{ravishanker1990differential, ravishanker2001differential, komaki2006shrinkage, tanaka2008superharmonic, barbaresco2006information, barbaresco2012information, tanaka2018superharmonic, choi2015kahlerian, choi2015geometric, oda2021shrinkage,choi2021kahlerian}, Bayesian prediction \cite{komaki2006shrinkage, tanaka2008superharmonic,tanaka2018superharmonic, choi2015kahlerian, choi2015geometric, oda2021shrinkage}, and more recently to stochastic processes with non-Gaussian features \cite{choi2025information}.
		
	Recently, tempered stable processes, a subclass of L\'evy processes characterized by heavy tails and jumps, were investigated from the perspective of information geometry \cite{choi2025information}. The $\alpha$-divergence between two tempered stable processes was derived as a two-fold generalization of the Kullback--Leibler divergence obtained by Kim and Lee \cite{kim2007relative} for classical tempered stable processes. Moreover, the corresponding Fisher information matrix and $\alpha$-connection, which capture the geometric structure of these processes, were constructed from the $\alpha$-divergence. Potential functions, canonical divergences, and dually flat structures were also investigated. Several statistical applications of this information geometry were considered.

	The information geometry developed for tempered stable processes is also relevant to financial applications because these processes are widely used in financial modeling. They are used to describe characteristic features of asset returns such as asymmetry and fat tails. Their applications include time series analysis via ARMA-GARCH models with tempered stable distributed innovations \cite{kim2009computing, kim2010tempered, kim2011time}, as well as portfolio management, risk management, and option pricing \cite{tsuchida2012mean, beck2013empirical, choi2015reward, choi2024diversified, georgiev2015periodic, anand2017equity, kim2023deep}.

	However, tempered stable processes constitute only a subset of the broader class of L\'evy processes. L\'evy processes include not only various tempered stable processes but also stable processes, compound Poisson processes, variance gamma processes, and other models \cite{cont2004nonparametric, rachev2011financial, madan1998variance}. These processes are fundamental in the study of jump processes and infinite divisibility. This motivates extending the information geometry developed for tempered stable processes \cite{choi2025information} to the broader class of L\'evy processes.
	
	In this paper, we develop the information geometry of L\'evy processes by deriving the $\alpha$-divergence and obtaining the corresponding Fisher information matrix and $\alpha$-connection. This extends the Kullback--Leibler divergence of L\'evy processes obtained by Cont and Tankov \cite{cont2004nonparametric} and the information geometry of tempered stable processes studied in Choi \cite{choi2025information}. Additionally, we obtain a potential function and dually flat geometry for L\'evy processes satisfying the affine condition. The $1$- and $(-1)$-divergences are also represented by the canonical and dual canonical divergences, respectively. Statistical applications to bias reduction and Bayesian predictive priors are also discussed.

	We apply the general results to various L\'evy processes applicable to financial modeling. For tempered stable processes and variance gamma processes, we obtain the potential functions and dual coordinates from the affine condition and show that their geometries are dually flat. The Merton model is considered as an example with a nonzero diffusion component, and its $\alpha$-divergence, Fisher information matrix, and $\alpha$-connection are derived from the general results.

	This paper is organized as follows. In Section~\ref{sec_levy}, we revisit the fundamentals of L\'evy processes and the relevant prior work. Section~\ref{sec_ig_levy} develops the information geometry of L\'evy processes and also discusses its statistical applications. In Section~\ref{sec_app}, we present several illustrative examples with financial applications. Finally, we conclude the paper.
	
\section{L\'evy processes}
\label{sec_levy}
	In this section, we provide the theoretical background on L\'evy processes \cite{sato1999levy, jacod2003limit}. We also review prior work on the Kullback--Leibler divergence for L\'evy processes \cite{cont2004nonparametric}. These foundations will be extended in the next Section~to investigate the information geometry of L\'evy processes.
	
	L\'evy processes are stochastic processes characterized by stationary and independent increments. They include Brownian motion and Poisson processes as special cases and have been widely applied to modeling discontinuities, jumps, and heavy tails in finance, physics, signal processing, and biology, where non-Gaussian and discontinuous behavior is frequently observed.
		
	Mathematically, a L\'evy process $(X_t, \mathbb{P})_{t \in [0,T]}$ on a probability space $(\Omega, \mathcal{F})$ is a stochastic process that satisfies the following properties:
	\begin{itemize}
		\item $X_0 = 0$ almost surely.
		\item \textbf{Independent increments:} For any $0 \le t_0 \le t_1 \le \cdots \le t_n$, the increments $X_{t_1} - X_{t_0}, X_{t_2} - X_{t_1},\cdots, X_{t_n} - X_{t_{n-1}}$ are mutually independent.
		\item \textbf{Stationary increments:} For all $s, t \ge 0$, the distribution of $X_{t+s} - X_s$ depends only on $t$.
		\item \textbf{Stochastic continuity:} For all $\epsilon > 0$, $\lim_{h \to 0} \mathbb{P}(|X_{t+h} - X_t| > \epsilon) = 0$.
	\end{itemize}
	These conditions ensure that L\'evy processes are useful for modeling phenomena with both continuous and discontinuous sample paths. While Brownian motion accounts for continuous evolution, L\'evy processes incorporate jumps via their L\'evy measures, capturing more complex and realistic dynamics.
	
	The L\'evy--It\^{o} decomposition states that a L\'evy process can be uniquely expressed as the combination of a Brownian (diffusion) component, a drift term, and a pure jump component governed by a L\'evy measure. This decomposition for a L\'evy process is fully described by its characteristic (L\'evy) triplet $(\sigma, \nu, \gamma)$, where $\sigma \ge 0$ is the diffusion coefficient, $\nu$ is the L\'evy measure on $\mathbb{R}$, and $\gamma \in \mathbb{R}$ is the drift term. The L\'evy measure $\nu$ is a positive measure satisfying the following conditions:
	\begin{align}
		\nu(\{0\}) = 0, \quad \int_{-1}^{+1} x^2 \nu(dx) < \infty, \quad \int_{|x| > 1} \nu(dx) < \infty.
	\end{align}
	
	A cornerstone of L\'evy process theory is the L\'evy--Khintchine formula, which characterizes L\'evy processes via their characteristic functions. According to the L\'evy--Khintchine formula, the characteristic function of a L\'evy process $(X_t, \mathbb{P})_{t \in [0, T]}$ with a L\'evy triplet $(\sigma, \nu, \gamma)$ is given by
	\begin{align}
	\label{charac_ftn_levy}
		\phi_t(z; \boldsymbol{\xi}) = \mathbb{E}[{\rm e}^{izX_t}] = \exp\left(-t \frac{\sigma^2 z^2}{2} + it \gamma z + t \int_{-\infty}^{\infty} ({\rm e}^{izx} - 1 - izx \mathbbm{1}_{|x| \leq 1}) \nu(dx)\right),
	\end{align}
	where $\boldsymbol{\xi}$ is the parameter vector of the L\'evy process, and $\mathbbm{1}$ is the indicator function such that $\mathbbm{1}_A(x)=1$ when $x\in A$ and 0 otherwise. 

	The logarithmic characteristic function $\Phi_t(z; \boldsymbol{\xi})$, also known as the L\'evy (or characteristic) exponent, is given by the L\'evy--Khintchine representation:
	\begin{align}
	\label{log_charac_ftn_levy}
		\Phi_t(z; \boldsymbol{\xi}) = \log \phi_t(z; \boldsymbol{\xi}) = -t \frac{\sigma^2 z^2}{2} + it \gamma z + t \int_{-\infty}^{\infty} ({\rm e}^{izx} - 1 - izx \mathbbm{1}_{|x| \leq 1}) \nu(dx).
	\end{align}
	It is noteworthy that the martingale condition for L\'evy processes is expressed as
	\begin{align}
		\Phi_t(-i; \boldsymbol{\xi}) = 0,
	\end{align}
	which is identically satisfied when
	\begin{align}
		\label{exp_levy_martingale}
		\gamma = -\frac{\sigma^2}{2} - \int_{-\infty}^{\infty} \left({\rm e}^{x} - 1 - x\mathbbm{1}_{|x| \leq 1}\right) \nu(dx).
	\end{align}
	This martingale condition is particularly important in the context of exponential L\'evy processes \cite{cont2004nonparametric}.
	
	According to Sato \cite{sato1999levy}, the Radon--Nikodym derivative between two L\'evy processes can be defined under some conditions. The following proposition outlines the formulation of the Radon--Nikodym derivative for L\'evy processes.
	\begin{prp}[Sato (1999)]
	\label{prp_sato}
		Let $(X_t, \mathbb{P})_{t\in[0,T]}$ and $(X_t, \mathbb{Q})_{t\in[0,T]}$ be two L\'evy processes on $(\Omega,\mathcal{F})$ with characteristic triplets $(\sigma_\mathbb{P},\nu_{\mathbb{P}},\gamma_{\mathbb{P}})$ and $(\sigma_\mathbb{Q},\nu_{\mathbb{Q}},\gamma_{\mathbb{Q}})$, respectively. Then $\mathbb{P}|_{\mathcal{F}_t}$ and $\mathbb{Q}|_{\mathcal{F}_t}$ are mutually absolutely continuous for all $t$ if and only if the following conditions are satisfied:
	\begin{itemize}
		\item $\sigma_\mathbb{P}=\sigma_\mathbb{Q}=\sigma$.
		\item The L\'evy measures are mutually absolutely continuous with
		\begin{align}
			\int_{-\infty}^{\infty}({\rm e}^{\psi(x)/2}-1)^2\nu_{\mathbb{Q}}(dx)<\infty,
		\end{align}
		where $\psi$ is the logarithm of the Radon--Nikodym density of $\nu_{\mathbb{P}}$ with respect to $\nu_{\mathbb{Q}}$: ${\rm e}^{\psi(x)}=d\nu_{\mathbb{P}}/d\nu_{\mathbb{Q}}$.
		\item If $\sigma=0$, then in addition, $\gamma_{\mathbb{Q}}$ must satisfy
		\begin{align}
			\label{zv_condition}
			\gamma_{\mathbb{P}}-\gamma_{\mathbb{Q}}=\int_{-1}^{1}x(\nu_{\mathbb{P}}-\nu_{\mathbb{Q}})(dx).
		\end{align}
	\end{itemize}
	
	The Radon--Nikodym derivative is given by
	\begin{align}
		\frac{d\mathbb{P}|_{\mathcal{F}_t}}{d\mathbb{Q}|_{\mathcal{F}_t}}={\rm e}^{U_t},
	\end{align}
	where $(U_t,\mathbb{U})_{t\in[0,T]}$ is a L\'evy process with a characteristic triplet $(\sigma_\mathbb{U},\nu_\mathbb{U},\gamma_{\mathbb{U}})$ given by 
	\begin{align}
		\label{u_sigma}
		\sigma_\mathbb{U}&=\sigma \eta,\\
		\label{u_nu}
		\nu_\mathbb{U}&=\nu_{\mathbb{Q}}\circ\psi^{-1},\\
		\label{u_gamma}
		\gamma_\mathbb{U}&=-\frac{1}{2}(\sigma\eta)^2-\int_{-\infty}^{\infty}({\rm e}^y-1-y\mathbbm{1}_{|y|\leq1})(\nu_{\mathbb{Q}}\circ\psi^{-1})(dy),
	\end{align}
	and $\eta$ is chosen so that
	\begin{align}
		\label{eta}
		\gamma_{\mathbb{P}}-\gamma_{\mathbb{Q}}-\int_{-1}^{1}x(\nu_{\mathbb{P}}-\nu_{\mathbb{Q}})(dx)=\sigma^2\eta.
	\end{align}
	Moreover, $U_t$ satisfies 
	\begin{align}
		\mathbb{E}_{\mathbb{Q}}[{\rm e}^{U_t}]=1
	\end{align}
	for all $t$.
	\end{prp}

	We assume that L\'evy processes in this paper satisfy Proposition~\ref{prp_sato}. Given the Radon--Nikodym derivative and the conditions stated in Proposition~\ref{prp_sato}, Cont and Tankov \cite{cont2004nonparametric} derived an explicit expression for the Kullback--Leibler divergence of L\'evy processes.
	\begin{prp}[Cont and Tankov (2004)]
	\label{prp_cont}
		Let $(X_t, \mathbb{P})_{t\in[0,T]}$ and $(X_t, \mathbb{Q})_{t\in[0,T]}$ be two L\'evy processes on $(\Omega,\mathcal{F})$ with characteristic triplets $(\sigma,\nu_{\mathbb{P}},\gamma_{\mathbb{P}})$ and $(\sigma,\nu_{\mathbb{Q}},\gamma_{\mathbb{Q}})$. The Kullback--Leibler divergence $KL(\mathbb{P}||\mathbb{Q})$ is then given by
		\begin{align}
			KL(\mathbb{P}||\mathbb{Q})=\frac{T}{2\sigma^2}\big(\gamma_{\mathbb{P}}-\gamma_{\mathbb{Q}}-\int_{-1}^{1}x(\nu_{\mathbb{P}}-\nu_{\mathbb{Q}})(dx)\big)^2+T\int_{-\infty}^{\infty} \Big(\Big(\frac{d\nu_{\mathbb{P}}}{d\nu_{\mathbb{Q}}}\Big) \log{\Big(\frac{d\nu_{\mathbb{P}}}{d\nu_{\mathbb{Q}}}\Big)}- \Big(\frac{d\nu_{\mathbb{P}}}{d\nu_{\mathbb{Q}}}\Big)+1\Big) \nu_{\mathbb{Q}}(dx).
		\end{align}
		
		If $\mathbb{P}$ and $\mathbb{Q}$ correspond to risk-neutral exponential L\'evy models, i.e., $\mathbb{P}$ and $\mathbb{Q}$ satisfy the martingale condition given by Eq.~(\ref{exp_levy_martingale}), the Kullback--Leibler divergence is reduced to
		\begin{align}
			KL(\mathbb{P}||\mathbb{Q})=\frac{T}{2\sigma^2}\big(\int_{-\infty}^{\infty}({\rm e}^x-1)(\nu_{\mathbb{P}}-\nu_{\mathbb{Q}})(dx)\big)^2+T\int_{-\infty}^{\infty} \Big(\Big(\frac{d\nu_{\mathbb{P}}}{d\nu_{\mathbb{Q}}}\Big) \log{\Big(\frac{d\nu_{\mathbb{P}}}{d\nu_{\mathbb{Q}}}\Big)}- \Big(\frac{d\nu_{\mathbb{P}}}{d\nu_{\mathbb{Q}}}\Big)+1\Big) \nu_{\mathbb{Q}}(dx).
		\end{align}
	\end{prp}
	In brief, the proof of Proposition~\ref{prp_cont} relies on applying the Radon--Nikodym derivative of Proposition~\ref{prp_sato} to the definition of the Kullback--Leibler divergence. We will revisit this proof in the following section.

\section{Information geometry of L\'evy processes}
\label{sec_ig_levy}
	In this section, we explore the information geometry of L\'evy processes. We first derive the $\alpha$-divergence of L\'evy processes from the definition of $f$-divergence and then investigate the associated geometric structures, including the Fisher information matrix, the $\alpha$-connection, potential, canonical divergence, and dually flat structure. Statistical applications are discussed afterward.
	
	Let us begin with the definition of $f$-divergence. The $f$-divergence between two probability distributions $\mathbb{P}$ and $\mathbb{Q}$ with density functions $p(x; \boldsymbol{\xi})$ and $q(x; \tilde{\boldsymbol{\xi}})$, respectively, is defined as
	\begin{align}
	\label{f_div}
		D_{f}(\mathbb{P}||\mathbb{Q})=\mathbb{E}_{\mathbb{Q}}\Big[f\Big(\frac{p(x;\boldsymbol{\xi})}{q(x;\tilde{\boldsymbol{\xi}})}\Big)\Big],
	\end{align}
	where $f$ is a convex function such that $f(1) = 0$ and $f(x)$ is finite for all $x > 0$.

	The $f$-divergence covers a variety of distance and divergence measures commonly used in statistics and information theory. For example, the Kullback--Leibler divergence arises from the $f$-divergence when $f$ is given by
	\begin{align}
	\label{kl_div_f}
		f_{KL}(t)=t \log{t} - t +1.
	\end{align}

	Another important example of $f$-divergence is the $\alpha$-divergence, which plays a central role in information geometry \cite{amari2000methods}. The $\alpha$-divergence can be expressed as the $f$-divergence with the following choice of $f$:
	\begin{align}
	\label{a_div_f}
		f^{(\alpha)}(t)=\left\{ 
	\begin{array}{ll}
	\frac{4}{1-\alpha ^{2}}\Big(\frac{1-\alpha}{2} t +\frac{1+\alpha}{2}-t^{\frac{1-\alpha}{2}}\Big) & (\alpha \neq \pm 1)\\ 
	t \log{t} - t +1 & (\alpha =-1)\\
	- \log{t} +t - 1 & (\alpha =1)
	\end{array}
	\right..
	\end{align}
	It is straightforward to verify that the functions $f^{(\alpha)}$ for $\alpha = \pm1$ can be obtained by applying L’H\^{o}pital’s rule to the $f$-divergence for $\alpha \neq \pm1$ and taking the limits $\alpha \to \pm1$, respectively. Moreover, the $\alpha = -1$ case in Eq.~(\ref{a_div_f}) coincides with Eq.~(\ref{kl_div_f}) for the Kullback--Leibler divergence.
	
	By applying Eq.~(\ref{a_div_f}) to Eq.~(\ref{f_div}), the $\alpha$-divergence between two probability density functions $p(x; \boldsymbol{\xi})$ and $q(x; \tilde{\boldsymbol{\xi}})$ is defined as follows \cite{amari2000methods}:
	\begin{align}
	 \label{alpha_div}
	D^{(\alpha )}(\mathbb{P}||\mathbb{Q})=\left\{ 
	\begin{array}{ll}
	\frac{4}{1-\alpha^{2}}\int_{-\infty}^{\infty}(\frac{1-\alpha}{2}p(x;\boldsymbol{\xi})+\frac{1+\alpha}{2}q(x;\tilde{\boldsymbol{\xi}})-p(x;\boldsymbol{\xi})^{\frac{1-\alpha}{2}}q(x;\tilde{\boldsymbol{\xi}})^{\frac{1+\alpha}{2}})dx & (\alpha \neq \pm 1)\\ 
	\int_{-\infty}^{\infty}(p(x;\boldsymbol{\xi}) \log{\frac{p(x;\boldsymbol{\xi})}{q(x;\tilde{\boldsymbol{\xi}})}}-p(x;\boldsymbol{\xi})+q(x;\tilde{\boldsymbol{\xi}}))dx & (\alpha =-1)\\
	\int_{-\infty}^{\infty}(q(x;\tilde{\boldsymbol{\xi}}) \log{\frac{q(x;\tilde{\boldsymbol{\xi}})}{p(x;\boldsymbol{\xi})}}-q(x;\tilde{\boldsymbol{\xi}})+p(x;\boldsymbol{\xi}))dx & (\alpha =1)
	\end{array}
	\right.,
	\end{align}
	where $\boldsymbol{\xi}$ and $\tilde{\boldsymbol{\xi}}$ are parameters of $p$ and $q$, respectively. 
	
	It is noteworthy that the $\alpha$-divergence exhibits an interesting property known as $\alpha$-duality \cite{amari2000methods}:
	\begin{align}
	\label{a_duality}
		D^{(\alpha )}(\mathbb{P}||\mathbb{Q})=D^{(-\alpha )}(\mathbb{Q}||\mathbb{P}).
	\end{align}
	Changing the sign of $\alpha$ is equivalent to exchanging the roles of $\mathbb{P}$ and $\mathbb{Q}$. In particular, the $0$-divergence, self-dual under the $\alpha$-duality, is proportional to the squared Hellinger distance.
	
	It is natural to extend the preceding discussion from probability density functions to Radon--Nikodym derivatives. The $f$-divergence can be defined in terms of the Radon--Nikodym derivative as follows:
	\begin{align}
	\label{f_div_rn}
		D_{f}(\mathbb{P}||\mathbb{Q})=\mathbb{E}_{\mathbb{Q}}\Big[f\Big(\frac{d\mathbb{P}}{d\mathbb{Q}}\Big)\Big].
	\end{align}

	Based on Eq.~(\ref{f_div_rn}), Cont and Tankov \cite{cont2004nonparametric} computed the Kullback--Leibler divergence for L\'evy processes under Proposition~\ref{prp_sato} by evaluating the following expression:
	\begin{align}
		KL(\mathbb{P}||\mathbb{Q})=\mathbb{E}_{\mathbb{Q}}[f_{KL}({\rm{e}}^{U_T})],
	\end{align}
	where $f_{KL}$ is given by Eq.~(\ref{kl_div_f}) and $U_T$ is defined in Proposition~\ref{prp_sato}. This calculation forms the basis for the proof of Proposition~\ref{prp_cont}, as mentioned earlier.
	
	Analogous to Proposition~\ref{prp_cont}, the $\alpha$-divergence between L\'evy processes can be derived by applying Eq.~(\ref{a_div_f}) to Eq.~(\ref{f_div_rn}):
	\begin{align}
	\label{a_div_rn}
		D^{(\alpha)}(\mathbb{P}||\mathbb{Q})=\mathbb{E}_{\mathbb{Q}}[f^{(\alpha)}({\rm{e}}^{U_T})].
	\end{align}
	
	Before proceeding with the derivation of the $\alpha$-divergence, we first consider the following proposition for further discussion.
	\begin{prp}
	\label{prp_delta}
	Let $(X_t, \mathbb{P})_{t\in[0,T]}$ and $(X_t, \mathbb{Q})_{t\in[0,T]}$ be L\'evy processes on $(\Omega,\mathcal{F})$ with L\'evy triplets $(\sigma,\nu_{\mathbb{P}},\gamma_{\mathbb{P}})$ and $(\sigma,\nu_{\mathbb{Q}},\gamma_{\mathbb{Q}})$, respectively. Let us define $\Delta^{(\alpha)}_T(\mathbb{P}||\mathbb{Q})$ as
	\begin{align}
		\label{delta_for_char_fn}
		\Delta^{(\alpha)}_T(\mathbb{P}||\mathbb{Q})=-\Phi_T(-i\frac{1-\alpha}{2};\boldsymbol{\xi}_{\mathbb{U}}),
	\end{align}
 	where $\Phi_T(z;\boldsymbol{\xi})$ is given by Eq.~(\ref{log_charac_ftn_levy}) and $\boldsymbol{\xi}_{\mathbb{U}}$ is the parameter vector of $U_T$ defined in Proposition~\ref{prp_sato}.

	When $\sigma\neq0$, we obtain $\Delta^{(\alpha)}_T(\mathbb{P}||\mathbb{Q})$ as
	\begin{align}
	\begin{aligned}
	\label{delta_nz}
		\Delta^{(\alpha)}_T(\mathbb{P}||\mathbb{Q})=&\frac{1-\alpha^2}{4}\frac{T}{2\sigma^2}\big(\gamma_{\mathbb{P}}-\gamma_{\mathbb{Q}}-\int_{-1}^{1}x(\nu_{\mathbb{P}}-\nu_{\mathbb{Q}})(dx)\big)^2\\
		&+T\int_{-\infty}^{\infty} \Big(\frac{1-\alpha}{2}\Big(\frac{d\nu_{\mathbb{P}}}{d\nu_{\mathbb{Q}}}\Big) +\frac{1+\alpha}{2}-\Big(\frac{d\nu_{\mathbb{P}}}{d\nu_{\mathbb{Q}}}\Big)^{\frac{1-\alpha}{2}}\Big) \nu_{\mathbb{Q}}(dx).
	\end{aligned}
	\end{align}
	
	If $\mathbb{P}$ and $\mathbb{Q}$ correspond to risk-neutral exponential L\'evy models, i.e., $\mathbb{P}$ and $\mathbb{Q}$ satisfy the martingale condition given by Eq.~(\ref{exp_levy_martingale}), $\Delta^{(\alpha)}_T(\mathbb{P}||\mathbb{Q})$ is represented with
	\begin{align}
	\begin{aligned}
	\label{delta_nz_martingale}
		\Delta^{(\alpha)}_T(\mathbb{P}||\mathbb{Q})=&\frac{1-\alpha^2}{4}\frac{T}{2\sigma^2}\big(\int_{-\infty}^{\infty}({\rm e}^{x}-1)(\nu_{\mathbb{P}}-\nu_{\mathbb{Q}})(dx)\big)^2\\ 
		&+T\int_{-\infty}^{\infty} \Big(\frac{1-\alpha}{2}\Big(\frac{d\nu_{\mathbb{P}}}{d\nu_{\mathbb{Q}}}\Big) +\frac{1+\alpha}{2}-\Big(\frac{d\nu_{\mathbb{P}}}{d\nu_{\mathbb{Q}}}\Big)^{\frac{1-\alpha}{2}}\Big) \nu_{\mathbb{Q}}(dx).
	\end{aligned}
	\end{align}
	
	When $\sigma=0$, $\Delta^{(\alpha)}_T(\mathbb{P}||\mathbb{Q})$ is given by
	\begin{align}
	\begin{aligned}
	\label{delta_zv}
		\Delta^{(\alpha)}_T(\mathbb{P}||\mathbb{Q})=T\int_{-\infty}^{\infty} \Big(\frac{1-\alpha}{2}\Big(\frac{d\nu_{\mathbb{P}}}{d\nu_{\mathbb{Q}}}\Big) +\frac{1+\alpha}{2}-\Big(\frac{d\nu_{\mathbb{P}}}{d\nu_{\mathbb{Q}}}\Big)^{\frac{1-\alpha}{2}}\Big) \nu_{\mathbb{Q}}(dx).
	\end{aligned}
	\end{align}
	\end{prp}

	\begin{proof}
	 By applying Eqs.~(\ref{u_sigma})--(\ref{u_gamma}), $T^{-1} \Phi_T\left(-i \frac{1 - \alpha}{2}; \boldsymbol{\xi}_{\mathbb{U}}\right)$ can be expressed as
	 \begin{align}
	 \label{delta_derivation}
	 \begin{aligned}
	 	&T^{-1}\Phi_T(-i\frac{1-\alpha}{2};\boldsymbol{\xi}_{\mathbb{U}})\\
		&=\Big(\frac{1-\alpha}{2}\Big)^2\frac{\sigma_{\mathbb{U}}^2}{2}+\frac{1-\alpha}{2}\gamma_{\mathbb{U}}+\int_{-\infty}^{\infty}({\rm e}^{\frac{1-\alpha}{2}y}-1-\frac{1-\alpha}{2}y \mathbbm{1}_{|y|\leq1})\nu_{\mathbb{U}}(dy)\\
		&=\Big(\frac{1-\alpha}{2}\Big)^2\frac{(\sigma\eta)^2}{2}+\frac{1-\alpha}{2}\Big(-\frac{(\sigma\eta)^2}{2}-\int_{-\infty}^{\infty}({\rm e}^y-1-y\mathbbm{1}_{|y|\leq1})(\nu_{\mathbb{Q}}\circ\psi^{-1})(dy)\Big)\\
		&\quad\textrm{ }+\int_{-\infty}^{\infty}({\rm e}^{\frac{1-\alpha}{2}y}-1-\frac{1-\alpha}{2}y \mathbbm{1}_{|y|\leq1})(\nu_{\mathbb{Q}}\circ\psi^{-1})(dy)\\
		&=-\frac{1-\alpha^2}{4}\frac{(\sigma\eta)^2}{2}-\int_{-\infty}^{\infty}(\frac{1-\alpha}{2}{\rm e}^y+\frac{1+\alpha}{2}-{\rm e}^{\frac{1-\alpha}{2}y})(\nu_{\mathbb{Q}}\circ\psi^{-1})(dy)\\
		&=-\frac{1-\alpha^2}{4}\frac{(\sigma\eta)^2}{2}-\int_{-\infty}^{\infty} \Big(\frac{1-\alpha}{2}\Big(\frac{d\nu_{\mathbb{P}}}{d\nu_{\mathbb{Q}}}\Big) +\frac{1+\alpha}{2}-\Big(\frac{d\nu_{\mathbb{P}}}{d\nu_{\mathbb{Q}}}\Big)^{\frac{1-\alpha}{2}}\Big) \nu_{\mathbb{Q}}(dx).
	\end{aligned}
	\end{align}
	
	When $\sigma\neq0$, plugging Eq.~(\ref{eta}) into Eq.~(\ref{delta_derivation}) yields $T^{-1} \Phi_T\left(-i \frac{1 - \alpha}{2}; \boldsymbol{\xi}_{\mathbb{U}}\right)$ as
	\begin{align}
		&T^{-1}\Phi_T(-i\frac{1-\alpha}{2};\boldsymbol{\xi}_{\mathbb{U}})\nonumber\\
		&=-\frac{1-\alpha^2}{4}\frac{1}{2\sigma^2}\big(\gamma_{\mathbb{P}}-\gamma_{\mathbb{Q}}-\int_{-1}^{1}x(\nu_{\mathbb{P}}-\nu_{\mathbb{Q}})(dx)\big)^2-\int_{-\infty}^{\infty} \Big(\frac{1-\alpha}{2}\Big(\frac{d\nu_{\mathbb{P}}}{d\nu_{\mathbb{Q}}}\Big) +\frac{1+\alpha}{2}-\Big(\frac{d\nu_{\mathbb{P}}}{d\nu_{\mathbb{Q}}}\Big)^{\frac{1-\alpha}{2}}\Big) \nu_{\mathbb{Q}}(dx).\nonumber
	\end{align}
	From the equation above, we obtain Eq.~(\ref{delta_nz}).
	
	If $\mathbb{P}$ and $\mathbb{Q}$ are risk-neutral exponential L\'evy models, the martingale condition in Eq.~(\ref{exp_levy_martingale}) is satisfied. By applying this condition to the first term in Eq.~(\ref{delta_nz}), the expression simplifies to
	\begin{align}
		&\gamma_{\mathbb{P}}-\gamma_{\mathbb{Q}}-\int_{-1}^{1}x(\nu_{\mathbb{P}}-\nu_{\mathbb{Q}})(dx)\nonumber\\
		&=\big(-\frac{\sigma^2}{2}-\int_{-\infty}^{\infty}({\rm e}^{x}-1-x \mathbbm{1}_{|x|\leq1})\nu_{\mathbb{P}}(dx)\big)-\big(-\frac{\sigma^2}{2}-\int_{-\infty}^{\infty}({\rm e}^{x}-1-x \mathbbm{1}_{|x|\leq1})\nu_{\mathbb{Q}}(dx)\big)-\int_{-1}^{1}x(\nu_{\mathbb{P}}-\nu_{\mathbb{Q}})(dx)\nonumber\\
		&=-\int_{-\infty}^{\infty}({\rm e}^{x}-1)(\nu_{\mathbb{P}}-\nu_{\mathbb{Q}})(dx).\nonumber
	\end{align}
	It is straightforward to verify that substituting the expression above into Eq.~(\ref{delta_nz}) yields Eq.~(\ref{delta_nz_martingale}).

	When $\sigma=0$ in Eq.~(\ref{delta_derivation}), $\Phi_T(-i\frac{1-\alpha}{2};\boldsymbol{\xi}_{\mathbb{U}})$ is obtained as
	\begin{align}
		T^{-1}\Phi_T(-i\frac{1-\alpha}{2};\boldsymbol{\xi}_{\mathbb{U}})&=-\int_{-\infty}^{\infty} \Big(\frac{1-\alpha}{2}\Big(\frac{d\nu_{\mathbb{P}}}{d\nu_{\mathbb{Q}}}\Big) +\frac{1+\alpha}{2}-\Big(\frac{d\nu_{\mathbb{P}}}{d\nu_{\mathbb{Q}}}\Big)^{\frac{1-\alpha}{2}}\Big) \nu_{\mathbb{Q}}(dx).\nonumber
	\end{align}
	From the equation above, we derive Eq.~(\ref{delta_zv}). This also implies that when $\sigma=0$, the first terms in both Eqs.~(\ref{delta_nz}) and (\ref{delta_nz_martingale}) vanish, allowing Eq.~(\ref{delta_zv}) to be obtained directly from them.
	\end{proof}

	It is noteworthy that $\Delta^{(\alpha)}_T(\mathbb{P} || \mathbb{Q})$ vanishes when $\mathbb{P}$ and $\mathbb{Q}$ are identical. Additionally, it also becomes zero when $\alpha = \pm 1$.
	
	We can consider $\Delta^{(\alpha)}_T(\mathbb{P}||\mathbb{Q})$ as the change of measure between two L\'evy processes. Let us rewrite the logarithmic Radon--Nikodym derivative between the L\'evy measures:
	\begin{align}
	\label{levy_reference_psi}
		\psi_{\mathbb{Q}}(x;\mathbb{P})=\log{\Big(\frac{d\nu_{\mathbb{P}}}{d\nu_{\mathbb{Q}}}(x)\Big)}=\psi(x).
	\end{align}
	We also define $c_{\mathbb{Q}}(\mathbb{P})$ by
	\begin{align}
	\label{levy_reference_c}
		c_{\mathbb{Q}}(\mathbb{P})=\left\{
		\begin{array}{ll}
		\frac{1}{\sigma}\big(\gamma_{\mathbb{P}}-\gamma_{\mathbb{Q}}-\int_{-1}^{1}x(\nu_{\mathbb{P}}-\nu_{\mathbb{Q}})(dx)\big) & (\sigma\neq0),\\
		0 & (\sigma=0).
		\end{array}
		\right.
	\end{align}
	For risk-neutral exponential L\'evy models with $\sigma\neq0$, the martingale condition gives
	\begin{align}
	\label{risk_neutral_c}
		c_{\mathbb{Q}}(\mathbb{P})=-\frac{1}{\sigma}\int_{-\infty}^{\infty}({\rm e}^x-1)(\nu_{\mathbb{P}}-\nu_{\mathbb{Q}})(dx).
	\end{align}
	For a parametric family with $\mathbb{P}=\mathbb{P}_{\boldsymbol{\xi}}$ and $\mathbb{Q}=\mathbb{P}_{\tilde{\boldsymbol{\xi}}}$, we write $\psi_{\tilde{\boldsymbol{\xi}}}(x;\boldsymbol{\xi})=\psi_{\mathbb{Q}}(x;\mathbb{P})$ and $c_{\tilde{\boldsymbol{\xi}}}(\boldsymbol{\xi})=c_{\mathbb{Q}}(\mathbb{P})$. For a fixed reference process $\mathbb{P}_*$ in the same family, we write $\psi_*(x;\boldsymbol{\xi})=\psi_{\mathbb{P}_*}(x;\mathbb{P}_{\boldsymbol{\xi}})$ and $c_*(\boldsymbol{\xi})=c_{\mathbb{P}_*}(\mathbb{P}_{\boldsymbol{\xi}})$. These quantities satisfy
	\begin{align}
	\label{levy_reference_difference}
		\psi_{\tilde{\boldsymbol{\xi}}}(x;\boldsymbol{\xi})&=\psi_*(x;\boldsymbol{\xi})-\psi_*(x;\tilde{\boldsymbol{\xi}})=-\psi_{\boldsymbol{\xi}}(x;\tilde{\boldsymbol{\xi}}),\\
		c_{\tilde{\boldsymbol{\xi}}}(\boldsymbol{\xi})&=c_*(\boldsymbol{\xi})-c_*(\tilde{\boldsymbol{\xi}})=-c_{\boldsymbol{\xi}}(\tilde{\boldsymbol{\xi}}).
	\end{align}
	
	Using Eqs.~(\ref{levy_reference_psi}) and~(\ref{levy_reference_c}), Proposition~\ref{prp_delta} takes the compact form:
	\begin{align}
	\label{delta_change_measure}
		\Delta^{(\alpha)}_T(\mathbb{P}||\mathbb{Q})=\frac{1-\alpha^2}{8}T c_{\mathbb{Q}}(\mathbb{P})^2+T\int_{-\infty}^{\infty}\Big(\frac{1-\alpha}{2}{\rm e}^{\psi_{\mathbb{Q}}(x;\mathbb{P})}+\frac{1+\alpha}{2}-{\rm e}^{\frac{1-\alpha}{2}\psi_{\mathbb{Q}}(x;\mathbb{P})}\Big)\nu_{\mathbb{Q}}(dx).
	\end{align}
	By Eq.~(\ref{levy_reference_c}), the first term in Eq.~(\ref{delta_change_measure}) vanishes when $\sigma=0$, so the same expression covers both cases in Proposition~\ref{prp_delta}.

	The following theorem provides the expression for the $\alpha$-divergence of L\'evy processes.
	\begin{thm}
	\label{thm_div_levy}
	Let $(X_t, \mathbb{P})_{t\in[0,T]}$ and $(X_t, \mathbb{Q})_{t\in[0,T]}$ be L\'evy processes on $(\Omega,\mathcal{F})$ with L\'evy triplets $(\sigma,\nu_{\mathbb{P}},\gamma_{\mathbb{P}})$ and $(\sigma,\nu_{\mathbb{Q}},\gamma_{\mathbb{Q}})$, respectively. When $\sigma\neq0$, we have the $\alpha$-divergence between two L\'evy processes as
	 \begin{align}
	 \label{a_div_measure_levy}
	D^{(\alpha )}(\mathbb{P}||\mathbb{Q})=\left\{ 
	\begin{array}{ll}
	\frac{4}{1-\alpha ^{2}}\Big\{1-\exp\big[-\frac{1-\alpha^2}{4}\frac{T}{2\sigma^2}\big(\gamma_{\mathbb{P}}-\gamma_{\mathbb{Q}}-\int_{-1}^{1}x(\nu_{\mathbb{P}}-\nu_{\mathbb{Q}})(dx)\big)^2 & \\
	-T\int_{-\infty}^{\infty}\Big(\frac{1-\alpha}{2}\Big(\frac{d\nu_{\mathbb{P}}}{d\nu_{\mathbb{Q}}}\Big) +\frac{1+\alpha}{2}-\Big(\frac{d\nu_{\mathbb{P}}}{d\nu_{\mathbb{Q}}}\Big)^{\frac{1-\alpha}{2}}\Big) \nu_{\mathbb{Q}}(dx)\big]\Big\} & (\alpha \neq \pm 1)\\ 
	\frac{T}{2\sigma^2}\big(\gamma_{\mathbb{P}}-\gamma_{\mathbb{Q}}-\int_{-1}^{1}x(\nu_{\mathbb{P}}-\nu_{\mathbb{Q}})(dx)\big)^2 & \\
	+T\int_{-\infty}^{\infty}\Big(\Big(\frac{d\nu_{\mathbb{P}}}{d\nu_{\mathbb{Q}}}\Big) \log{\Big(\frac{d\nu_{\mathbb{P}}}{d\nu_{\mathbb{Q}}}\Big)}- \Big(\frac{d\nu_{\mathbb{P}}}{d\nu_{\mathbb{Q}}}\Big)+1\Big) \nu_{\mathbb{Q}}(dx) & (\alpha =-1)\\
	\frac{T}{2\sigma^2}\big(\gamma_{\mathbb{Q}}-\gamma_{\mathbb{P}}-\int_{-1}^{1}x(\nu_{\mathbb{Q}}-\nu_{\mathbb{P}})(dx)\big)^2 & \\
	+T\int_{-\infty}^{\infty}\Big(\Big(\frac{d\nu_{\mathbb{Q}}}{d\nu_{\mathbb{P}}}\Big) \log{\Big(\frac{d\nu_{\mathbb{Q}}}{d\nu_{\mathbb{P}}}\Big)}- \Big(\frac{d\nu_{\mathbb{Q}}}{d\nu_{\mathbb{P}}}\Big)+1\Big) \nu_{\mathbb{P}}(dx) & (\alpha =1)
	\end{array}
	\right..
	\end{align}
	
	If $\mathbb{P}$ and $\mathbb{Q}$ correspond to risk-neutral exponential L\'evy models, i.e., $\mathbb{P}$ and $\mathbb{Q}$ satisfy the martingale condition given by Eq.~(\ref{exp_levy_martingale}), the $\alpha$-divergence between two L\'evy processes is represented with
	\begin{align}
	 \label{a_div_measure_levy_martingale}
	D^{(\alpha )}(\mathbb{P}||\mathbb{Q})=\left\{ 
	\begin{array}{ll}
	\frac{4}{1-\alpha ^{2}}\Big\{1-\exp\big[-\frac{1-\alpha^2}{4}\frac{T}{2\sigma^2}\big(\int_{-\infty}^{\infty}({\rm e}^{x}-1)(\nu_{\mathbb{P}}-\nu_{\mathbb{Q}})(dx)\big)^2 & \\
	-T\int_{-\infty}^{\infty}\Big(\frac{1-\alpha}{2}\Big(\frac{d\nu_{\mathbb{P}}}{d\nu_{\mathbb{Q}}}\Big) +\frac{1+\alpha}{2}-\Big(\frac{d\nu_{\mathbb{P}}}{d\nu_{\mathbb{Q}}}\Big)^{\frac{1-\alpha}{2}}\Big) \nu_{\mathbb{Q}}(dx)\big]\Big\} & (\alpha \neq \pm 1)\\ 
	\frac{T}{2\sigma^2}\big(\int_{-\infty}^{\infty}({\rm e}^{x}-1)(\nu_{\mathbb{P}}-\nu_{\mathbb{Q}})(dx)\big)^2 & \\
	+T\int_{-\infty}^{\infty}\Big(\Big(\frac{d\nu_{\mathbb{P}}}{d\nu_{\mathbb{Q}}}\Big) \log{\Big(\frac{d\nu_{\mathbb{P}}}{d\nu_{\mathbb{Q}}}\Big)}- \Big(\frac{d\nu_{\mathbb{P}}}{d\nu_{\mathbb{Q}}}\Big)+1\Big) \nu_{\mathbb{Q}}(dx) & (\alpha =-1)\\
	\frac{T}{2\sigma^2}\big(\int_{-\infty}^{\infty}({\rm e}^{x}-1)(\nu_{\mathbb{Q}}-\nu_{\mathbb{P}})(dx)\big)^2 & \\
	+T\int_{-\infty}^{\infty}\Big(\Big(\frac{d\nu_{\mathbb{Q}}}{d\nu_{\mathbb{P}}}\Big) \log{\Big(\frac{d\nu_{\mathbb{Q}}}{d\nu_{\mathbb{P}}}\Big)}- \Big(\frac{d\nu_{\mathbb{Q}}}{d\nu_{\mathbb{P}}}\Big)+1\Big) \nu_{\mathbb{P}}(dx) & (\alpha =1)
	\end{array}
	\right..
	\end{align}
	
	When $\sigma=0$, the $\alpha$-divergence between two L\'evy processes is given by
\begin{align}
	\label{a_div_measure_levy_zv}
	D^{(\alpha )}(\mathbb{P}||\mathbb{Q})=\left\{ 
	\begin{array}{ll}
		\frac{4}{1-\alpha ^{2}}\Big\{1-\exp\big[-T\int_{-\infty}^{\infty}\Big(\frac{1-\alpha}{2}\Big(\frac{d\nu_{\mathbb{P}}}{d\nu_{\mathbb{Q}}}\Big) +\frac{1+\alpha}{2}-\Big(\frac{d\nu_{\mathbb{P}}}{d\nu_{\mathbb{Q}}}\Big)^{\frac{1-\alpha}{2}}\Big) \nu_{\mathbb{Q}}(dx)\big]\Big\} & (\alpha \neq \pm 1)\\ 
	T\int_{-\infty}^{\infty}\Big(\Big(\frac{d\nu_{\mathbb{P}}}{d\nu_{\mathbb{Q}}}\Big) \log{\Big(\frac{d\nu_{\mathbb{P}}}{d\nu_{\mathbb{Q}}}\Big)}- \Big(\frac{d\nu_{\mathbb{P}}}{d\nu_{\mathbb{Q}}}\Big)+1\Big) \nu_{\mathbb{Q}}(dx) & (\alpha =-1)\\
	T\int_{-\infty}^{\infty}\Big(\Big(\frac{d\nu_{\mathbb{Q}}}{d\nu_{\mathbb{P}}}\Big) \log{\Big(\frac{d\nu_{\mathbb{Q}}}{d\nu_{\mathbb{P}}}\Big)}- \Big(\frac{d\nu_{\mathbb{Q}}}{d\nu_{\mathbb{P}}}\Big)+1\Big) \nu_{\mathbb{P}}(dx) & (\alpha =1)
	\end{array}
	\right..
	\end{align}
	\end{thm}
	\begin{proof}
	We consider the cases $\sigma \neq 0$ and $\sigma = 0$ separately.
	
	We begin the proof with the case $\sigma \neq 0$. First, the $\alpha$-divergence for $\alpha \neq \pm 1$ is derived by evaluating Eq.~(\ref{a_div_rn}):
	\begin{align}
		D^{(\alpha)}(\mathbb{P}||\mathbb{Q})&=\frac{4}{1-\alpha ^{2}}\mathbb{E}_{\mathbb{Q}}[\frac{1-\alpha}{2}{\rm{e}}^{U_T}+\frac{1+\alpha}{2}-{\rm{e}}^{\frac{1-\alpha}{2}U_T}].\nonumber
	\end{align}
	By using the logarithmic characteristic function $\Phi_t(z; \boldsymbol{\xi})$ from Eq.~(\ref{log_charac_ftn_levy}), the $\alpha$-divergence for $\alpha \neq \pm 1$ can be expressed with
	\begin{align}
	\label{a_div_delta}
		D^{(\alpha)}(\mathbb{P}||\mathbb{Q})=\frac{4}{1-\alpha ^{2}}\Big(1-\mathbb{E}_{\mathbb{Q}}[{\rm{e}}^{\frac{1-\alpha}{2}U_T}]\Big)=\frac{4}{1-\alpha ^{2}}\Big(1-{\rm{e}}^{\Phi_T(-i\frac{1-\alpha}{2};\boldsymbol{\xi}_{\mathbb{U}})}\Big).
	\end{align}
	According to Proposition~\ref{prp_delta}, the $\alpha$-divergence in Eq.~(\ref{a_div_delta}) can be rewritten in terms of $\Delta^{(\alpha)}_T(\mathbb{P}|| \mathbb{Q})$ as	
	\begin{align}
	\label{a_div_in_delta}
		D^{(\alpha)}(\mathbb{P}||\mathbb{Q})=\frac{4}{1-\alpha ^{2}}\Big(1-{\rm{e}}^{-\Delta^{(\alpha)}_T(\mathbb{P}||\mathbb{Q})}\Big).
	\end{align}
	By substituting Eq.~(\ref{delta_nz}) into Eq.~(\ref{a_div_in_delta}), the $\alpha$-divergence for $\alpha\neq\pm1$ is given by
	\begin{align}
		D^{(\alpha)}(\mathbb{P}||\mathbb{Q})=\frac{4}{1-\alpha ^{2}}\Big\{1-&\exp\big[-\frac{1-\alpha^2}{4}\frac{T}{2\sigma^2}\big(\gamma_{\mathbb{P}}-\gamma_{\mathbb{Q}}-\int_{-1}^{1}x(\nu_{\mathbb{P}}-\nu_{\mathbb{Q}})(dx)\big)^2\nonumber\\
		&-T\int_{-\infty}^{\infty} \Big(\frac{1-\alpha}{2}\Big(\frac{d\nu_{\mathbb{P}}}{d\nu_{\mathbb{Q}}}\Big) +\frac{1+\alpha}{2}-\Big(\frac{d\nu_{\mathbb{P}}}{d\nu_{\mathbb{Q}}}\Big)^{\frac{1-\alpha}{2}}\Big) \nu_{\mathbb{Q}}(dx)\big]\Big\}.\nonumber
	\end{align}
	
	When $\alpha = -1$, the Kullback--Leibler divergence can be obtained in two different ways. The first approach is a direct calculation by plugging the $\alpha = -1$ case in Eq.~(\ref{a_div_f}) into Eq.~(\ref{a_div_rn}), which gives the same result with the derivation presented by Cont and Tankov \cite{cont2004nonparametric}. The second method involves applying L’H\^{o}pital’s rule to the expression for $\alpha$-divergence with $\alpha \neq \pm1$ and taking the limit $\alpha \to -1$, which also recovers the Kullback--Leibler divergence.
	
	We begin with the first solution given by
	\begin{align}
		&D^{(-1)}(\mathbb{P}||\mathbb{Q})=\mathbb{E}_{\mathbb{Q}}[U_T{\rm{e}}^{U_T}-{\rm{e}}^{U_T}+1]=\mathbb{E}_{\mathbb{Q}}[U_T{\rm{e}}^{U_T}]\nonumber\\
		&=-i\frac{d}{dz}\mathbb{E}_{\mathbb{Q}}[{\rm{e}}^{izU_T}]|_{z=-i}=-i\Phi'_T(-i;\boldsymbol{\xi}_{\mathbb{U}}){\rm{e}}^{\Phi_T(-i;\boldsymbol{\xi}_{\mathbb{U}})}=-i\Phi'_T(-i;\boldsymbol{\xi}_{\mathbb{U}})\nonumber,
	\end{align}
	where $\Phi'_T(z;\boldsymbol{\xi}_{\mathbb{U}})$ is directly obtained by differentiating Eq.~(\ref{log_charac_ftn_levy}):
	\begin{align}
		\Phi'_T(z;\boldsymbol{\xi}_{\mathbb{U}})=-T\sigma^2_{\mathbb{U}}z+iT\gamma_{\mathbb{U}}+T\int_{-\infty}^{\infty}(iy{\rm e}^{izy}-iy \mathbbm{1}_{|y|\leq1})\nu_{\mathbb{U}}(dy).\nonumber
	\end{align}
	Using Eq.~(\ref{u_sigma})--(\ref{eta}), $\Phi'_T(-i;\boldsymbol{\xi}_{\mathbb{U}})$ is obtained as
	\begin{align}
		&-iT^{-1}\Phi'_T(-i;\boldsymbol{\xi}_{\mathbb{U}})=\sigma^2_{\mathbb{U}}+\gamma_{\mathbb{U}}+\int_{-\infty}^{\infty}(y{\rm e}^{y}-y\mathbbm{1}_{|y|\leq1})\nu_{\mathbb{U}}(dy)\nonumber\\
		&=(\sigma\eta)^2+\Big(-\frac{(\sigma\eta)^2}{2}-\int_{-\infty}^{\infty}({\rm e}^y-1-y \mathbbm{1}_{|y|\leq1})(\nu_{\mathbb{Q}}\circ\psi^{-1})(dy)\Big)+\int_{-\infty}^{\infty}(y{\rm e}^{y}-y \mathbbm{1}_{|y|\leq1})(\nu_{\mathbb{Q}}\circ\psi^{-1})(dy)\nonumber\\
		&=\frac{(\sigma\eta)^2}{2}+\int_{-\infty}^{\infty}(y{\rm e}^{y}-{\rm e}^{y}+1)(\nu_{\mathbb{Q}}\circ\psi^{-1})(dy)\nonumber\\
		&=\frac{(\sigma\eta)^2}{2}+\int_{-\infty}^{\infty} \Big(\Big(\frac{d\nu_{\mathbb{P}}}{d\nu_{\mathbb{Q}}}\Big)\log{\Big(\frac{d\nu_{\mathbb{P}}}{d\nu_{\mathbb{Q}}}\Big)}+1-\Big(\frac{d\nu_{\mathbb{P}}}{d\nu_{\mathbb{Q}}}\Big)\Big) \nu_{\mathbb{Q}}(dx)\nonumber\\
		&=\frac{1}{2\sigma^2}\big(\gamma_{\mathbb{P}}-\gamma_{\mathbb{Q}}-\int_{-1}^{1}x(\nu_{\mathbb{P}}-\nu_{\mathbb{Q}})(dx)\big)^2+\int_{-\infty}^{\infty} \Big(\Big(\frac{d\nu_{\mathbb{P}}}{d\nu_{\mathbb{Q}}}\Big)\log{\Big(\frac{d\nu_{\mathbb{P}}}{d\nu_{\mathbb{Q}}}\Big)}+1-\Big(\frac{d\nu_{\mathbb{P}}}{d\nu_{\mathbb{Q}}}\Big)\Big) \nu_{\mathbb{Q}}(dx).\nonumber
	\end{align}
	From the calculation above, the Kullback--Leibler divergence can be expressed as
	\begin{align}
		D^{(-1)}(\mathbb{P}||\mathbb{Q})=\frac{T}{2\sigma^2}\big(\gamma_{\mathbb{P}}-\gamma_{\mathbb{Q}}-\int_{-1}^{1}x(\nu_{\mathbb{P}}-\nu_{\mathbb{Q}})(dx)\big)^2+T\int_{-\infty}^{\infty} \Big(\Big(\frac{d\nu_{\mathbb{P}}}{d\nu_{\mathbb{Q}}}\Big)\log{\Big(\frac{d\nu_{\mathbb{P}}}{d\nu_{\mathbb{Q}}}\Big)}+1-\Big(\frac{d\nu_{\mathbb{P}}}{d\nu_{\mathbb{Q}}}\Big)\Big) \nu_{\mathbb{Q}}(dx).\nonumber
	\end{align}
	The same result can also be obtained by applying L’H\^{o}pital’s rule.

	Similarly, the $1$-divergence can be obtained by applying L’H\^{o}pital’s rule or by direct calculation from the definition of the $\alpha$-divergence. Additionally, a third approach is based on the $\alpha$-duality:
	\begin{align}
		&D^{(1)}(\mathbb{P}||\mathbb{Q})=D^{(-1)}(\mathbb{Q}||\mathbb{P})\nonumber\\
		&=\frac{T}{2\sigma^2}\big(\gamma_{\mathbb{Q}}-\gamma_{\mathbb{P}}-\int_{-1}^{1}x(\nu_{\mathbb{Q}}-\nu_{\mathbb{P}})(dx)\big)^2+T\int_{-\infty}^{\infty} \Big(\Big(\frac{d\nu_{\mathbb{Q}}}{d\nu_{\mathbb{P}}}\Big)\log{\Big(\frac{d\nu_{\mathbb{Q}}}{d\nu_{\mathbb{P}}}\Big)}+1-\Big(\frac{d\nu_{\mathbb{Q}}}{d\nu_{\mathbb{P}}}\Big)\Big) \nu_{\mathbb{P}}(dx).\nonumber
	\end{align}
	It is straightforward to verify that the same result is obtained from the other two methods: applying L’H\^{o}pital’s rule and performing a direct calculation based on the definition of $\alpha$-divergence.

	We now return to the case of $\sigma = 0$. As noted earlier, when $\sigma = 0$, the first terms in Eqs.~(\ref{delta_nz}) and (\ref{delta_nz_martingale}) vanish by Proposition~\ref{prp_delta} due to Eq.~(\ref{zv_condition}) and Eq.~(\ref{eta}), allowing Eq.~(\ref{a_div_measure_levy_zv}) to be derived directly from Eqs.~(\ref{a_div_measure_levy}) and (\ref{a_div_measure_levy_martingale}):
	\begin{align}
		D^{(\alpha)}(\mathbb{P}||\mathbb{Q})=\frac{4}{1-\alpha ^{2}}\Big\{1-\exp\big[-T\int_{-\infty}^{\infty} \Big(\frac{1-\alpha}{2}\Big(\frac{d\nu_{\mathbb{P}}}{d\nu_{\mathbb{Q}}}\Big) +\frac{1+\alpha}{2}-\Big(\frac{d\nu_{\mathbb{P}}}{d\nu_{\mathbb{Q}}}\Big)^{\frac{1-\alpha}{2}}\Big) \nu_{\mathbb{Q}}(dx)\big]\Big\}.\nonumber
	\end{align}
	Similarly, the Kullback--Leibler divergence can be obtained as
	\begin{align}
		D^{(-1)}(\mathbb{P}||\mathbb{Q})=T\int_{-\infty}^{\infty} \Big(\frac{d\nu_{\mathbb{P}}}{d\nu_{\mathbb{Q}}}\log{\Big(\frac{d\nu_{\mathbb{P}}}{d\nu_{\mathbb{Q}}}\Big)}+1-\Big(\frac{d\nu_{\mathbb{P}}}{d\nu_{\mathbb{Q}}}\Big)\Big) \nu_{\mathbb{Q}}(dx),\nonumber
	\end{align}
	and the corresponding dual divergence is given by
	\begin{align}
		D^{(1)}(\mathbb{P}||\mathbb{Q})=T\int_{-\infty}^{\infty} \Big(\frac{d\nu_{\mathbb{Q}}}{d\nu_{\mathbb{P}}}\log{\Big(\frac{d\nu_{\mathbb{Q}}}{d\nu_{\mathbb{P}}}\Big)}+1-\Big(\frac{d\nu_{\mathbb{Q}}}{d\nu_{\mathbb{P}}}\Big)\Big) \nu_{\mathbb{P}}(dx).\nonumber
	\end{align}
	\end{proof}

	It is straightforward to check that the Kullback--Leibler divergences in Eqs.~(\ref{a_div_measure_levy}) and (\ref{a_div_measure_levy_martingale}), i.e., the cases with $\alpha = -1$, are identical to those presented in Proposition~\ref{prp_cont}, as derived by Cont and Tankov \cite{cont2004nonparametric}.
	
	Similarly, using Eqs.~(\ref{levy_reference_psi}) and~(\ref{levy_reference_c}), Theorem~\ref{thm_div_levy} can be written as
	\begin{align}
	\label{a_div_change_measure}
	D^{(\alpha)}(\mathbb{P}||\mathbb{Q})=\left\{
	\begin{array}{ll}
	\frac{4}{1-\alpha^2}\Big\{1-\exp\big[-\frac{1-\alpha^2}{8}T c_{\mathbb{Q}}(\mathbb{P})^2 & \\
	\qquad\qquad-T\int_{-\infty}^{\infty}\Big(\frac{1-\alpha}{2}{\rm e}^{\psi_{\mathbb{Q}}(x;\mathbb{P})}+\frac{1+\alpha}{2}-{\rm e}^{\frac{1-\alpha}{2}\psi_{\mathbb{Q}}(x;\mathbb{P})}\Big)\nu_{\mathbb{Q}}(dx)\big]\Big\} & (\alpha\neq\pm1),\\
	\frac{T}{2}c_{\mathbb{Q}}(\mathbb{P})^2+T\int_{-\infty}^{\infty}\Big(\psi_{\mathbb{Q}}(x;\mathbb{P}){\rm e}^{\psi_{\mathbb{Q}}(x;\mathbb{P})}-{\rm e}^{\psi_{\mathbb{Q}}(x;\mathbb{P})}+1\Big)\nu_{\mathbb{Q}}(dx) & (\alpha=-1),\\
	\frac{T}{2}c_{\mathbb{P}}(\mathbb{Q})^2+T\int_{-\infty}^{\infty}\Big(\psi_{\mathbb{P}}(x;\mathbb{Q}){\rm e}^{\psi_{\mathbb{P}}(x;\mathbb{Q})}-{\rm e}^{\psi_{\mathbb{P}}(x;\mathbb{Q})}+1\Big)\nu_{\mathbb{P}}(dx) & (\alpha=1).
	\end{array}
	\right.
	\end{align}
	By Eq.~(\ref{levy_reference_c}), Eq.~(\ref{a_div_change_measure}) covers both $\sigma\neq0$ and $\sigma=0$.
	
	The following Corollary states that Theorem~\ref{thm_div_levy} can be expressed in a simpler form by using $\Delta^{(\alpha)}_T(\mathbb{P} || \mathbb{Q})$ defined in Proposition~\ref{prp_delta}.
	\begin{crl}
	\label{crl_div_levy_delta}
	Let $(X_t, \mathbb{P})_{t\in[0,T]}$ and $(X_t, \mathbb{Q})_{t\in[0,T]}$ be L\'evy processes on $(\Omega,\mathcal{F})$ with L\'evy triplets $(\sigma,\nu_{\mathbb{P}},\gamma_{\mathbb{P}})$ and $(\sigma,\nu_{\mathbb{Q}},\gamma_{\mathbb{Q}})$, respectively. The $\alpha$-divergence between the two L\'evy processes can be represented in terms of $\Delta^{(\alpha)}_T(\mathbb{P}||\mathbb{Q})$ as
	\begin{align}
	 \label{a_div_measure_levy_in_delta}
	D^{(\alpha )}(\mathbb{P}||\mathbb{Q})=\left\{ 
	\begin{array}{ll}
	\frac{4}{1-\alpha ^{2}}\Big(1-{\rm{e}}^{-\Delta^{(\alpha)}_T(\mathbb{P}||\mathbb{Q})}\Big) & (\alpha \neq \pm 1)\\ 
	\lim_{\alpha\to-1}-\frac{2\partial_{\alpha} \Delta^{(\alpha)}_T(\mathbb{P}||\mathbb{Q})}{\alpha} & (\alpha =-1)\\
	\lim_{\alpha\to1}-\frac{2\partial_{\alpha} \Delta^{(\alpha)}_T(\mathbb{P}||\mathbb{Q})}{\alpha} & (\alpha =1)
	\end{array}
	\right..
	\end{align}
	\end{crl}
	\begin{proof}
	The proof is straightforward. For $\alpha \neq \pm1$, the result follows directly from Eq.~(\ref{a_div_in_delta}). When $\alpha = \pm1$, as discussed in the proof of Theorem~\ref{thm_div_levy}, we apply L’H\^{o}pital’s rule to Eq.~(\ref{a_div_in_delta}) and take the limits $\alpha \to \pm1$:
	\begin{align}
		D^{(\pm1)}(\mathbb{P}||\mathbb{Q})=\lim_{\alpha\to\pm1}-\frac{2\partial_{\alpha} \Delta^{(\alpha)}_T(\mathbb{P}||\mathbb{Q})}{\alpha}{\rm e}^{-\Delta^{(\alpha)}_T(\mathbb{P}||\mathbb{Q})}.
	\end{align}
	Since $\Delta^{(\alpha)}_T(\mathbb{P} || \mathbb{Q})$ tends to zero as $\alpha \to \pm1$, the exponential term $\mathrm{e}^{-\Delta^{(\alpha)}_T(\mathbb{P} || \mathbb{Q})}$ approaches 1, and its deviation from 1 becomes negligible in the limits.
	\end{proof}

	For small values of $\Delta^{(\alpha)}_T(\mathbb{P} || \mathbb{Q})$, the $\alpha$-divergence for $\alpha \neq \pm1$ can be approximated by its linear form:
	\begin{align}
	\label{a_div_levy_linear}
		D^{(\alpha )}(\mathbb{P}||\mathbb{Q})\approx\frac{4}{1-\alpha ^{2}}\Delta^{(\alpha)}_T(\mathbb{P}||\mathbb{Q})
	\end{align}
	This approximation also provides the same values of the $\alpha$-divergence in the limits $\alpha \to \pm1$.

	Using the $\alpha$-divergence for L\'evy processes obtained above, we derive the associated information-geometric structures, including the Fisher information matrix and the $\alpha$-connection.
	\begin{thm}
	\label{thm_geo_levy}
	Let $(X_t, \mathbb{P})_{t\in[0,T]}$ be a L\'evy process on $(\Omega,\mathcal{F})$ with a L\'evy triplet $(\sigma,\nu,\gamma)$. When $\sigma\neq0$, the Fisher information matrix of L\'evy processes is represented with
	\begin{align}
		\label{metric_levy}
		g_{ij}=\frac{T}{\sigma^2}\partial_i \big(\gamma-\int_{-1}^{1}x\nu(dx)\big) \partial_j \big(\gamma-\int_{-1}^{1}x\nu(dx)\big)+T\int_{-\infty}^{\infty}\partial_i \log{\Big(\frac{d\nu}{dx}\Big)} \partial_j \log{\Big(\frac{d\nu}{dx}\Big)} \nu(dx),
	\end{align}
	and the $\alpha$-connection for information geometry of L\'evy processes is given by
	\begin{align}
		\label{conn_levy}
		\begin{split}
		\Gamma^{(\alpha)}_{ij,k}=&\frac{T}{\sigma^2}\partial_i \partial_j \big(\gamma-\int_{-1}^{1}x\nu(dx)\big) \partial_k \big(\gamma-\int_{-1}^{1}x\nu(dx)\big)\\
		&+T\int_{-\infty}^{\infty}\Big(\partial_i \partial_j \log{\Big(\frac{d\nu}{dx}\Big)}+\frac{1-\alpha}{2}\partial_i \log{\Big(\frac{d\nu}{dx}\Big)} \partial_j \log{\Big(\frac{d\nu}{dx}\Big)}\Big)\partial_k \log{\Big(\frac{d\nu}{dx}\Big)} \nu(dx),
		\end{split}
	\end{align}
	where $i,j,$ and $k$ index the coordinates $\boldsymbol{\xi}$.
	
	If $\mathbb{P}$ and $\mathbb{Q}$ correspond to risk-neutral exponential L\'evy models, i.e., $\mathbb{P}$ and $\mathbb{Q}$ satisfy the martingale condition given by Eq.~(\ref{exp_levy_martingale}), the Fisher information matrix of L\'evy processes is represented with
	\begin{align}
		\label{metric_levy_martingale}
		g_{ij}=&\frac{T}{\sigma^2}\partial_i \big(\int_{-\infty}^{\infty}({\rm e}^{x}-1)\nu(dx)\big) \partial_j \big(\int_{-\infty}^{\infty}({\rm e}^{x}-1)\nu(dx)\big)+T\int_{-\infty}^{\infty}\partial_i \log{\Big(\frac{d\nu}{dx}\Big)} \partial_j \log{\Big(\frac{d\nu}{dx}\Big)} \nu(dx),
	\end{align}
	and the $\alpha$-connection for information geometry of L\'evy processes is given by
	\begin{align}
		\label{conn_levy_martingale}
		\begin{split}
		\Gamma^{(\alpha)}_{ij,k}=&\frac{T}{\sigma^2}\partial_i \partial_j \big(\int_{-\infty}^{\infty}({\rm e}^{x}-1)\nu(dx)\big) \partial_k \big(\int_{-\infty}^{\infty}({\rm e}^{x}-1)\nu(dx)\big)\\
		&+T\int_{-\infty}^{\infty}\Big(\partial_i \partial_j \log{\Big(\frac{d\nu}{dx}\Big)}+\frac{1-\alpha}{2}\partial_i \log{\Big(\frac{d\nu}{dx}\Big)} \partial_j \log{\Big(\frac{d\nu}{dx}\Big)}\Big)\partial_k \log{\Big(\frac{d\nu}{dx}\Big)} \nu(dx).
		\end{split}
	\end{align}
	
	When $\sigma=0$, the Fisher information matrix and the $\alpha$-connection for information geometry of L\'evy processes are represented with
	\begin{align}
		\label{metric_levy_zv}
		g_{ij}=&T\int_{-\infty}^{\infty}\partial_i \log{\Big(\frac{d\nu}{dx}\Big)} \partial_j \log{\Big(\frac{d\nu}{dx}\Big)} \nu(dx),\\
		\label{conn_levy_zv}
		\Gamma^{(\alpha)}_{ij,k}=&T\int_{-\infty}^{\infty}\Big(\partial_i \partial_j \log{\Big(\frac{d\nu}{dx}\Big)}+\frac{1-\alpha}{2}\partial_i \log{\Big(\frac{d\nu}{dx}\Big)} \partial_j \log{\Big(\frac{d\nu}{dx}\Big)}\Big)\partial_k \log{\Big(\frac{d\nu}{dx}\Big)} \nu(dx).
	\end{align}

	\end{thm}
	
	\begin{proof}
	For a given divergence or distance function $\mathcal{D}$, the metric tensor and the connection of information geometry \cite{amari2000methods} are derived as
	\begin{align}
		\label{ig_metric}
		g_{ij}&=-\mathcal{D}(\partial_i, \tilde{\partial}_j)|_{\tilde{\boldsymbol{\xi}}=\boldsymbol{\xi}},\\
		\label{ig_connection}
		\Gamma_{ij,k}&=-\mathcal{D}(\partial_i \partial_j,\tilde{\partial}_k)|_{\tilde{\boldsymbol{\xi}}=\boldsymbol{\xi}},
	\end{align}
	where $i,j,$ and $k$ index the coordinates $\boldsymbol{\xi}$.
	
	As done in Theorem~\ref{thm_div_levy}, we divide the proof into two cases: $\sigma \neq 0$ and $\sigma = 0$. We begin with the case of $\sigma \neq 0$. By substituting Eq.~(\ref{a_div_measure_levy}) into Eqs.~(\ref{ig_metric}) and (\ref{ig_connection}), we obtain Eqs.~(\ref{metric_levy}) and (\ref{conn_levy}), respectively.
	
	When the martingale condition in Eq.~(\ref{exp_levy_martingale}) is applied to Eqs.~(\ref{metric_levy}) and (\ref{conn_levy}), the resulting geometric structures are given by Eqs.~(\ref{metric_levy_martingale}) and (\ref{conn_levy_martingale}). The same geometric results can also be obtained by plugging Eq.~(\ref{a_div_measure_levy_martingale}) into Eqs.~(\ref{ig_metric}) and (\ref{ig_connection}).

	In the case of $\sigma = 0$, the first terms of the exponents in Eqs.~(\ref{a_div_measure_levy}) and (\ref{a_div_measure_levy_martingale}) vanish. These terms also make the first terms in Eqs.~(\ref{metric_levy}) and (\ref{conn_levy}) (or equivalently, in Eqs.~(\ref{metric_levy_martingale}) and (\ref{conn_levy_martingale})) zero. A direct computation using Eq.~(\ref{a_div_measure_levy_zv}) in Eqs.~(\ref{ig_metric}) and (\ref{ig_connection}) provides Eqs.~(\ref{metric_levy_zv}) and (\ref{conn_levy_zv}), respectively.
	\end{proof}

	Eqs.~(\ref{levy_reference_psi}) and~(\ref{levy_reference_c}) also provide an equivalent representation of Theorem~\ref{thm_geo_levy}. Let us fix $\tilde{\boldsymbol{\xi}}$ and take $\mathbb{P}_{\tilde{\boldsymbol{\xi}}}$ as a reference process in the same parametric family. Since the reference process is fixed, the derivatives below are taken only with respect to $\boldsymbol{\xi}$. The Fisher information matrix and the $\alpha$-connection are given by
	\begin{align}
	\label{metric_levy_change_measure}
		g_{ij}={}&T\,\partial_i c_{\tilde{\boldsymbol{\xi}}}(\boldsymbol{\xi})\,\partial_j c_{\tilde{\boldsymbol{\xi}}}(\boldsymbol{\xi})
		+T\int_{-\infty}^{\infty}\partial_i\psi_{\tilde{\boldsymbol{\xi}}}(x;\boldsymbol{\xi})\,\partial_j\psi_{\tilde{\boldsymbol{\xi}}}(x;\boldsymbol{\xi})\,\nu_{\boldsymbol{\xi}}(dx),\\
	\label{conn_levy_change_measure}
		\Gamma^{(\alpha)}_{ij,k}={}&T\,\partial_i\partial_j c_{\tilde{\boldsymbol{\xi}}}(\boldsymbol{\xi})\,\partial_k c_{\tilde{\boldsymbol{\xi}}}(\boldsymbol{\xi})\nonumber\\
		&+T\int_{-\infty}^{\infty}\Big(\partial_i\partial_j\psi_{\tilde{\boldsymbol{\xi}}}(x;\boldsymbol{\xi})+\frac{1-\alpha}{2}\partial_i\psi_{\tilde{\boldsymbol{\xi}}}(x;\boldsymbol{\xi})\,\partial_j\psi_{\tilde{\boldsymbol{\xi}}}(x;\boldsymbol{\xi})\Big)\partial_k\psi_{\tilde{\boldsymbol{\xi}}}(x;\boldsymbol{\xi})\,\nu_{\boldsymbol{\xi}}(dx).
	\end{align}
	These expressions are independent of the choice of the fixed reference process. Let $\mathbb{P}_{\bar{\boldsymbol{\xi}}}$ be another fixed reference process. From Eq.~(\ref{levy_reference_difference}), $c_{\bar{\boldsymbol{\xi}}}(\boldsymbol{\xi})=c_{\tilde{\boldsymbol{\xi}}}(\boldsymbol{\xi})-c_{\tilde{\boldsymbol{\xi}}}(\bar{\boldsymbol{\xi}})$ and $\psi_{\bar{\boldsymbol{\xi}}}(x;\boldsymbol{\xi})=\psi_{\tilde{\boldsymbol{\xi}}}(x;\boldsymbol{\xi})-\psi_{\tilde{\boldsymbol{\xi}}}(x;\bar{\boldsymbol{\xi}})$. Since the second terms are independent of $\boldsymbol{\xi}$, all derivatives with respect to $\boldsymbol{\xi}$ in Eqs.~(\ref{metric_levy_change_measure}) and~(\ref{conn_levy_change_measure}) are unchanged by the choice of the fixed reference process. Since $\nu_{\boldsymbol{\xi}}$ is also independent of the reference process, the Fisher information matrix and the $\alpha$-connection are independent of this choice.

	When $\sigma=0$, the terms containing $c_{\tilde{\boldsymbol{\xi}}}$ vanish. If the L\'evy measures admit densities with respect to the Lebesgue measure, then $\partial_i\psi_{\tilde{\boldsymbol{\xi}}}(x;\boldsymbol{\xi})=\partial_i\log(d\nu_{\boldsymbol{\xi}}/dx)$, and Eqs.~(\ref{metric_levy_change_measure}) and~(\ref{conn_levy_change_measure}) reproduce Eqs.~(\ref{metric_levy}) and~(\ref{conn_levy}) for $\sigma\neq0$, and Eqs.~(\ref{metric_levy_zv}) and~(\ref{conn_levy_zv}) for $\sigma=0$. Under the martingale condition in Eq.~(\ref{exp_levy_martingale}), the same representation gives Eqs.~(\ref{metric_levy_martingale}) and~(\ref{conn_levy_martingale}).

	Similar to Corollary~\ref{crl_div_levy_delta}, the information geometry obtained in Theorem~\ref{thm_geo_levy} can also be expressed in terms of $\Delta^{(\alpha)}_T(\mathbb{P} || \mathbb{Q})$.
	\begin{crl}
	\label{crl_geo_levy_delta}
	Let $(X_t, \mathbb{P})_{t\in[0,T]}$ be a L\'evy process on $(\Omega,\mathcal{F})$ with L\'evy triplet $(\sigma,\nu,\gamma)$. The Fisher information matrix of L\'evy processes and the $\alpha$-connection for information geometry of L\'evy processes are given by
	\begin{align}
		\label{metric_levy_delta}
		g_{ij}&=-\frac{4}{1-\alpha ^{2}}\partial_i\tilde{\partial}_j\Delta^{(\alpha)}_T(\mathbb{P}||\mathbb{Q})|_{\tilde{\boldsymbol{\xi}}=\boldsymbol{\xi}},\\
		\label{conn_levy_delta}
		\Gamma_{ij,k}^{(\alpha)}&=-\frac{4}{1-\alpha ^{2}}\partial_i\partial_j\tilde{\partial}_k\Delta^{(\alpha)}_T(\mathbb{P}||\mathbb{Q})|_{\tilde{\boldsymbol{\xi}}=\boldsymbol{\xi}},
	\end{align}
	where $i,j,$ and $k$ index the coordinates $\boldsymbol{\xi}$.
	\end{crl}
	\begin{proof}
		By substituting the $\alpha$-divergence from Corollary~\ref{crl_div_levy_delta} into Eq.~(\ref{ig_metric}) and Eq.~(\ref{ig_connection}), we obtain the Fisher information matrix and the $\alpha$-connection for L\'evy processes as follows:
		\begin{align}
		g_{ij}&=-\frac{4}{1-\alpha ^{2}}\Delta^{(\alpha)}_T(\partial_i|| \tilde{\partial}_j){\rm e}^{-\Delta^{(\alpha)}_T(\mathbb{P}||\mathbb{Q})}|_{\tilde{\boldsymbol{\xi}}=\boldsymbol{\xi}},\\
		\Gamma_{ij,k}^{(\alpha)}&=-\frac{4}{1-\alpha ^{2}}\Delta^{(\alpha)}_T(\partial_i \partial_j||\tilde{\partial}_k){\rm e}^{-\Delta^{(\alpha)}_T(\mathbb{P}||\mathbb{Q})}|_{\tilde{\boldsymbol{\xi}}=\boldsymbol{\xi}},
		\end{align}
		because the following terms vanish:
		\begin{align}
			\Delta^{(\alpha)}_T(\partial_i || \cdot)|_{\tilde{\boldsymbol{\xi}}=\boldsymbol{\xi}}=0, \qquad \Delta^{(\alpha)}_T(\cdot || \tilde\partial_j)|_{\tilde{\boldsymbol{\xi}}=\boldsymbol{\xi}}=0.
		\end{align}
		As mentioned earlier, when $\mathbb{P}$ and $\mathbb{Q}$ are identical, i.e., $\boldsymbol{\xi} = \tilde{\boldsymbol{\xi}}$, we have $\Delta^{(\alpha)}_T(\mathbb{P} || \mathbb{Q}) = 0$. This simplifies the expressions for the Fisher information matrix and the $\alpha$-connection to Eqs.~(\ref{metric_levy_delta}) and (\ref{conn_levy_delta}).
	\end{proof}
	It is noteworthy that Corollary~\ref{crl_geo_levy_delta} extends continuously by taking the limits $\alpha \to \pm1$, since the factor $\frac{4}{1 - \alpha^2}$ cancels out during differentiation when computing the partial derivatives.

	We next derive a potential and prove the dually flat structure of the geometry in Theorem~\ref{thm_geo_levy}. The following theorem shows that under the affine condition on Eqs.~(\ref{levy_reference_psi}) and~(\ref{levy_reference_c}), a potential function generates the Fisher information matrix and the $\alpha$-connection, and the resulting statistical manifold is dually flat.

	\begin{thm}
	\label{thm_affine_levy}
	Suppose that $c_*(\boldsymbol{\xi})$ and $\psi_*(x;\boldsymbol{\xi})$ are affine in $\boldsymbol{\xi}$. A potential function $K_*$ for the information geometry of L\'evy processes is represented as
	\begin{align}
	\label{levy_affine_potential}
		K_*(\boldsymbol{\xi})=\frac{T}{2}c_*(\boldsymbol{\xi})^2+T\int_{-\infty}^{\infty}\Big({\rm e}^{\psi_*(x;\boldsymbol{\xi})}-1-\psi_*(x;\boldsymbol{\xi})\Big)\nu_*(dx),
	\end{align}
	such that the Fisher information matrix and the $\alpha$-connection are obtained from $K_*$:
	\begin{align}
	\label{levy_affine_hessian_metric}
		g_{ij}&=\partial_i\partial_jK_*(\boldsymbol{\xi}),\\
	\label{levy_affine_hessian_connection}
		\Gamma^{(\alpha)}_{ij,k}&=\frac{1-\alpha}{2}\partial_i\partial_j\partial_kK_*(\boldsymbol{\xi}).
	\end{align}
	With the dual coordinates represented as
	\begin{align}
	\label{levy_affine_dual_coordinate}
		\eta_i=\partial_iK_*(\boldsymbol{\xi}),
	\end{align}
	the resulting statistical manifold is dually flat.
	\end{thm}

	\begin{proof}
	Since $c_*(\boldsymbol{\xi})$ and $\psi_*(x;\boldsymbol{\xi})$ are affine in $\boldsymbol{\xi}$, their second derivatives vanish:
	\begin{align}
	\label{levy_affine_condition}
		\partial_i\partial_j c_*(\boldsymbol{\xi})=0,\qquad \partial_i\partial_j\psi_*(x;\boldsymbol{\xi})=0,
	\end{align}
	for all $i,j$ and $\nu_*$-almost every $x$.

	With Eq.~(\ref{levy_affine_condition}), differentiating Eq.~(\ref{levy_affine_potential}) and using $\nu_{\boldsymbol{\xi}}(dx)={\rm e}^{\psi_*(x;\boldsymbol{\xi})}\nu_*(dx)$ provides the following:
	\begin{align}
	\label{levy_affine_second_derivative}
		\partial_i\partial_jK_*&=T\partial_i c_*\,\partial_j c_*+T\int_{-\infty}^{\infty}\partial_i\psi_*\,\partial_j\psi_*\,\nu_{\boldsymbol{\xi}}(dx),\\
	\label{levy_affine_third_derivative}
		\partial_i\partial_j\partial_kK_*&=T\int_{-\infty}^{\infty}\partial_i\psi_*\,\partial_j\psi_*\,\partial_k\psi_*\,\nu_{\boldsymbol{\xi}}(dx).
	\end{align}
	Comparing with Eqs.~(\ref{metric_levy_change_measure}) and~(\ref{conn_levy_change_measure}), Eqs.~(\ref{levy_affine_second_derivative}) and~(\ref{levy_affine_third_derivative}) give Eqs.~(\ref{levy_affine_hessian_metric}) and~(\ref{levy_affine_hessian_connection}), respectively. When $\sigma=0$, the terms containing $c_*$ vanish by Eq.~(\ref{levy_reference_c}), and the same result follows.

	We next consider the dependence of the potential on the fixed reference process. Let $\mathbb{P}_{\tilde{\boldsymbol{\xi}}}$ be another fixed reference process. From Eq.~(\ref{levy_reference_difference}), the potential constructed from this reference satisfies
	\begin{align}
	\label{levy_potential_reference_change}
		K_{\tilde{\boldsymbol{\xi}}}(\boldsymbol{\xi})-K_*(\boldsymbol{\xi})={}-T c_*(\tilde{\boldsymbol{\xi}})c_*(\boldsymbol{\xi})+T\int_{-\infty}^{\infty}\Big(1-{\rm e}^{\psi_*(x;\tilde{\boldsymbol{\xi}})}\Big)\psi_*(x;\boldsymbol{\xi})\nu_*(dx)+A_{\tilde{\boldsymbol{\xi}}},
	\end{align}
	where $A_{\tilde{\boldsymbol{\xi}}}$ does not depend on $\boldsymbol{\xi}$. Since $c_*(\boldsymbol{\xi})$ and $\psi_*(x;\boldsymbol{\xi})$ are affine in $\boldsymbol{\xi}$, the right-hand side of Eq.~(\ref{levy_potential_reference_change}) is affine in $\boldsymbol{\xi}$. Its second and third derivatives vanish, and we obtain
	\begin{align}
	\begin{aligned}
	\label{levy_reference_derivative_invariance}
		\partial_i\partial_jK_{\tilde{\boldsymbol{\xi}}}(\boldsymbol{\xi})&=\partial_i\partial_jK_*(\boldsymbol{\xi}),\\
		\partial_i\partial_j\partial_kK_{\tilde{\boldsymbol{\xi}}}(\boldsymbol{\xi})&=\partial_i\partial_j\partial_kK_*(\boldsymbol{\xi}).
	\end{aligned}
	\end{align}
	From this, changing the fixed reference process changes the potential only by an affine function and does not change its second and third derivatives.

	When $\alpha=1$, the coefficient of the third derivative in Eq.~(\ref{levy_affine_hessian_connection}) vanishes:
	\begin{align}
		\Gamma^{(1)}_{ij,k}=0.
	\end{align}
	In this case, $\boldsymbol{\xi}$ is an affine coordinate system for the $1$-connection. At $\alpha=-1$, the connection coincides with the third derivative of $K_*$:
	\begin{align}
		\Gamma^{(-1)}_{ij,k}=\partial_i\partial_j\partial_kK_*.
	\end{align}
	The Jacobian of the dual coordinates in Eq.~(\ref{levy_affine_dual_coordinate}) coincides with the Fisher information matrix:
	\begin{align}
		\partial_j\eta_i=g_{ij}.
	\end{align}
	Since the Fisher information matrix is nonsingular, $\boldsymbol{\eta}$ forms a local coordinate system. The covariant second derivative of $\eta_a$ with respect to the $(-1)$-connection vanishes:
	\begin{align}
	\label{levy_affine_dual_flat_calculation}
		\partial_i\partial_j\eta_a
		-\Gamma^{(-1)k}_{ij}\partial_k\eta_a
		=\partial_i\partial_j\partial_aK_*
		-g^{k\ell}\partial_i\partial_j\partial_\ell K_*\,g_{ak}=0.
	\end{align}
	Eq.~(\ref{levy_affine_dual_flat_calculation}) shows that $\boldsymbol{\eta}$ is an affine coordinate system for the $(-1)$-connection. The $1$- and $(-1)$-connections also satisfy the duality relation with respect to the Fisher information matrix:
	\begin{align}
		\partial_k g_{ij}
		=\partial_i\partial_j\partial_kK_*
		=\Gamma^{(1)}_{ki,j}+\Gamma^{(-1)}_{kj,i}.
	\end{align}
	Both connections are flat, and the statistical manifold is dually flat \cite{amari2000methods}.
	\end{proof}

	The canonical divergences can also be expressed directly by using the potential function in Theorem~\ref{thm_affine_levy}.
	\begin{crl}
	\label{crl_affine_levy}
	Under the affine condition in Theorem~\ref{thm_affine_levy}, the Bregman divergence $B_{K_*}$ associated with the potential, with its two arguments corresponding to the primal and dual orientations, is defined as
	\begin{align}
	\label{levy_affine_bregman}
		B_{K_*}(\boldsymbol{\xi},\tilde{\boldsymbol{\xi}})=K_*(\boldsymbol{\xi})-K_*(\tilde{\boldsymbol{\xi}})	-\partial_iK_*(\tilde{\boldsymbol{\xi}})(\xi^i-\tilde\xi^i).
	\end{align}
	Additionally, the two orientations coincide with the $1$- and $(-1)$-divergences, respectively:
	\begin{align}
	\label{levy_affine_canonical}
		B_{K_*}(\boldsymbol{\xi},\tilde{\boldsymbol{\xi}})&=D^{(1)}(\mathbb{P}_{\boldsymbol{\xi}}||\mathbb{P}_{\tilde{\boldsymbol{\xi}}}),\\
		B_{K_*}(\tilde{\boldsymbol{\xi}},\boldsymbol{\xi})&=D^{(-1)}(\mathbb{P}_{\boldsymbol{\xi}}||\mathbb{P}_{\tilde{\boldsymbol{\xi}}}).
	\end{align}
	\end{crl}

	\begin{proof}
	Since $c_*$ and $\psi_*$ are affine, their increments along $\boldsymbol{\xi}-\tilde{\boldsymbol{\xi}}$ are linear:
	\begin{align}
		\partial_i c_*(\tilde{\boldsymbol{\xi}})(\xi^i-\tilde\xi^i)&=c_*(\boldsymbol{\xi})-c_*(\tilde{\boldsymbol{\xi}}),\\
		\partial_i\psi_*(x;\tilde{\boldsymbol{\xi}})(\xi^i-\tilde\xi^i)&=\psi_*(x;\boldsymbol{\xi})-\psi_*(x;\tilde{\boldsymbol{\xi}}).
	\end{align}
	For brevity, write $c_*=c_*(\boldsymbol{\xi})$, $\tilde c_*=c_*(\tilde{\boldsymbol{\xi}})$, $\psi_*=\psi_*(x;\boldsymbol{\xi})$, and $\tilde\psi_*=\psi_*(x;\tilde{\boldsymbol{\xi}})$. Substitution of Eq.~(\ref{levy_affine_potential}) into Eq.~(\ref{levy_affine_bregman}) gives
	\begin{align}
	\label{levy_affine_bregman_direct}
		B_{K_*}(\boldsymbol{\xi},\tilde{\boldsymbol{\xi}})=\frac{T}{2}(c_*-\tilde c_*)^2+T\int_{-\infty}^{\infty}\Big[{\rm e}^{\psi_*}-{\rm e}^{\tilde\psi_*}-{\rm e}^{\tilde\psi_*}(\psi_*-\tilde\psi_*)\Big]\nu_*(dx).
	\end{align}
	By Eq.~(\ref{levy_reference_difference}), $c_*-\tilde c_*=c_{\tilde{\boldsymbol{\xi}}}(\boldsymbol{\xi})$ and $\psi_*-\tilde\psi_*=\psi_{\tilde{\boldsymbol{\xi}}}(x;\boldsymbol{\xi})$. Using $\nu_{\tilde{\boldsymbol{\xi}}}={\rm e}^{\tilde\psi_*}\nu_*$ removes the reference measure from the jump integral:
	\begin{align}
	\label{levy_affine_bregman_pairwise}
		B_{K_*}(\boldsymbol{\xi},\tilde{\boldsymbol{\xi}})
		=\frac{T}{2}c_{\tilde{\boldsymbol{\xi}}}(\boldsymbol{\xi})^2+T\int_{-\infty}^{\infty}\Big({\rm e}^{\psi_{\tilde{\boldsymbol{\xi}}}(x;\boldsymbol{\xi})}-1-\psi_{\tilde{\boldsymbol{\xi}}}(x;\boldsymbol{\xi})\Big)\nu_{\tilde{\boldsymbol{\xi}}}(dx).
	\end{align}
	Eq.~(\ref{levy_affine_bregman_pairwise}) coincides with the $1$-divergence in Theorem~\ref{thm_div_levy}. For the reverse orientation, the calculation is identical except for the jump integrand:
	\begin{align}
		B_{K_*}(\tilde{\boldsymbol{\xi}},\boldsymbol{\xi})
		=\frac{T}{2}c_{\tilde{\boldsymbol{\xi}}}(\boldsymbol{\xi})^2+T\int_{-\infty}^{\infty}\Big(\psi_{\tilde{\boldsymbol{\xi}}}(x;\boldsymbol{\xi}){\rm e}^{\psi_{\tilde{\boldsymbol{\xi}}}(x;\boldsymbol{\xi})}-{\rm e}^{\psi_{\tilde{\boldsymbol{\xi}}}(x;\boldsymbol{\xi})}+1\Big)\nu_{\tilde{\boldsymbol{\xi}}}(dx).
	\end{align}
	This expression coincides with the $(-1)$-divergence. These two divergences prove Eq.~(\ref{levy_affine_canonical}).
	\end{proof}

	Additionally, the same potential represents the exponent $\Delta_T^{(\alpha)}$. For $\alpha\neq\pm1$, consider parameter pairs for which the argument of $K_*$ below remains in the parameter domain:
	\begin{align}
	\label{levy_affine_jensen}
		\Delta_T^{(\alpha)}(\mathbb{P}_{\boldsymbol{\xi}}||\mathbb{P}_{\tilde{\boldsymbol{\xi}}})
		=\frac{1-\alpha}{2}K_*(\boldsymbol{\xi})+\frac{1+\alpha}{2}K_*(\tilde{\boldsymbol{\xi}})-K_*\big(\frac{1-\alpha}{2}\boldsymbol{\xi}+\frac{1+\alpha}{2}\tilde{\boldsymbol{\xi}}\big).
	\end{align}
	Substitution of Eq.~(\ref{levy_affine_potential}) into Eq.~(\ref{levy_affine_jensen}) reproduces $\Delta_T^{(\alpha)}$ in Proposition~\ref{prp_delta}.

	The affine condition also has a direct interpretation in terms of a change of measure. In the Girsanov notation for semimartingales \cite{jacod2003limit}, $\beta$ and $Y$ denote the continuous and jump parameters, respectively. Eqs.~(\ref{levy_reference_psi}) and~(\ref{levy_reference_c}) correspond to $\psi_*=\log Y$ and $c_*=\sigma\beta$, respectively. If the coordinates are shifted so that the reference process corresponds to $\boldsymbol{\xi}=0$, then $c_*(0)=0$ and $\psi_*(x;0)=0$. Under the affine condition, they can therefore be written as $c_*(\boldsymbol{\xi})=b_i\xi^i$ and $\psi_*(x;\boldsymbol{\xi})=h_i(x)\xi^i$. Under the usual integrability conditions, the likelihood ratio then has the exponential representation studied for exponential families of stochastic processes \cite{kuchler1997exponential}. In particular, the natural exponential or Esscher family satisfies this affine condition. When $\sigma=0$, only the affine condition on $\psi_*$ is required.

	This information geometry of L\'evy processes can be applied to statistical problems such as bias reduction \cite{firth1993bias,kosmidis2009bias,kosmidis2010generic} and the construction of Bayesian predictive priors \cite{komaki2006shrinkage, tanaka2008superharmonic,tanaka2018superharmonic, choi2015kahlerian, choi2015geometric, oda2021shrinkage}. 

	Firth's bias reduction can be obtained by modifying the score equations for regular parametric models \cite{firth1993bias,kosmidis2009bias,kosmidis2010generic}. For full exponential families with canonical parametrization, the adjustment can be obtained from the penalized log-likelihood approach:
	\begin{align}
	\label{penalized_ll}
		l^{*}(\boldsymbol{\xi})=l(\boldsymbol{\xi})+\log \mathcal{J},
	\end{align}
	where $l(\boldsymbol{\xi})$ is the log-likelihood and $\mathcal{J}$ is the Jeffreys prior.

	For L\'evy process geometry, the Jeffreys prior is obtained as
	\begin{align}
		\mathcal{J} \propto (\det{g_{ij}})^{1/2}= \bigg(\det{\Big(-\frac{4}{1-\alpha ^{2}}\Delta^{(\alpha)}_T(\partial_i || \tilde{\partial}_j)|_{\tilde{\boldsymbol{\xi}}=\boldsymbol{\xi}}\Big)}\bigg)^{1/2}.
	\end{align}	

	This Jeffreys prior can also be employed as a baseline for constructing Bayesian predictive priors. According to Komaki \cite{komaki2006shrinkage}, a prior of the following form can have a smaller asymptotic predictive risk than the Jeffreys prior:
	\begin{align}
	\label{bayes_predict_prior}
		\tilde{\mathcal{J}}=\rho^2 \mathcal{J},
	\end{align}
	where $\rho$ is a positive nonconstant superharmonic function satisfying $\Delta \rho \le 0$. The Laplace--Beltrami operator $\Delta$ is defined as
	\begin{align}
		\Delta \rho = \frac{1}{ (\det{g_{ij}})^{1/2}}\partial_i \big((\det{g_{ij}})^{1/2} g^{ij}\partial_j \rho\big).
	\end{align}

\section{Examples}
\label{sec_app}
	In this section, we present applications of the results developed in Section~\ref{sec_ig_levy} to various L\'evy processes. In particular, the examples considered are closely related to models that are widely used in finance.
\subsection{Tempered stable processes}
	We start with the geometric structures of tempered stable processes, which form a subclass of L\'evy processes. Choi \cite{choi2025information} found the information geometry of various tempered stable processes and its statistical applications in parameter estimation and Bayesian prediction. We compare our results with those previously established in the literature.

	A tempered stable distribution is characterized by L\'evy triplet $(0, \nu, \gamma)$ \cite{rachev2011financial}, with the L\'evy measure $\nu$ given by
	\begin{align}
	\label{levy_measure_ts}
		\nu(dx;\boldsymbol{\xi})=t(x;\boldsymbol{\xi})\nu_{\textrm{stable}}(dx;\boldsymbol{\xi}),
	\end{align}
	where $t(x; \boldsymbol{\xi})$ denotes the tempering function of the tempered stable process, and $\nu_{\textrm{stable}}(dx; \boldsymbol{\xi})$ represents the L\'evy measure of a stable process. The L\'evy measure of stable processes is given by
	\begin{align}
	\label{levy_measure_stable}
		\nu_{\textrm{stable}}(dx;\boldsymbol{\xi})=\Big(\frac{C_+}{x^{a_++1}} \mathbbm{1}_{x>0}(x)+\frac{C_-}{|x|^{a_-+1}} \mathbbm{1}_{x<0}(x)\Big)dx,
	\end{align}
	where $C_{\pm}>0$, and $a_\pm\in(0,2)\setminus \{1\}$.

	Since tempered stable processes are a special case of L\'evy processes with $\sigma = 0$, their $\alpha$-divergence is directly given by Eq.~(\ref{a_div_measure_levy_zv}) in Theorem~\ref{thm_div_levy}. It is noteworthy that Eq.~(\ref{a_div_measure_levy_zv}) is identical to the $\alpha$-divergence presented in Choi \cite{choi2025information}. Similarly, the information geometry of tempered stable processes can be derived from Theorem~\ref{thm_geo_levy}. The Fisher information matrix and the $\alpha$-connection for the information geometry of tempered stable processes are given by Eqs.~(\ref{metric_levy_zv}) and (\ref{conn_levy_zv}) which are also identical to those in Choi \cite{choi2025information}.
		
	We next consider generalized tempered stable (GTS) processes and classical tempered stable (CTS) processes. Their information geometry was derived in Choi \cite{choi2025information} and we recover the same results directly from the general L\'evy-process formulas developed in Section~\ref{sec_ig_levy}. Since the CTS process is a special case of the GTS process, we present the GTS result first for each geometric object and then obtain the corresponding CTS result by imposing the CTS condition. 
	
	The GTS process uses the following tempering function \cite{rachev2011financial}:
	\begin{align}
	\label{tempering_fns_gts}
		t(x;\boldsymbol{\xi})=\big({\rm e}^{-\lambda_+ x} \mathbbm{1}_{x>0}(x)+{\rm e}^{-\lambda_- |x|} \mathbbm{1}_{x<0}(x)\big),
	\end{align}
	where $\lambda_{\pm}>0$. Plugging Eq.~(\ref{tempering_fns_gts}) to Eq.~(\ref{levy_measure_ts}) leads to the L\'evy measure of GTS processes:
	\begin{align}
	\label{levy_measure_gts}
		\nu(dx;\boldsymbol{\xi})=\Big(\frac{C_+{\rm e}^{-\lambda_+ x}}{x^{a_++1}} \mathbbm{1}_{x>0}(x)+\frac{C_-{\rm e}^{-\lambda_- |x|}}{|x|^{a_-+1}} \mathbbm{1}_{x<0}(x)\Big)dx.
	\end{align}
	
	The Radon--Nikodym derivative of GTS processes is defined as
	\begin{align}
	\label{radon_nikodym_gts}
		\frac{d\nu_{\mathbb{P}}}{d\nu_{\mathbb{Q}}}={\rm e}^{-(\lambda_+-\tilde{\lambda}_+)x} \mathbbm{1}_{x>0}(x)+{\rm e}^{-(\lambda_{-}-\tilde{\lambda}_{-})|x|} \mathbbm{1}_{x<0}(x).
	\end{align}
	It is noteworthy that all the parameters of two GTS processes, except for $\lambda_{\pm}$, are identical due to the existence condition of the Radon--Nikodym derivative in Proposition~\ref{prp_sato}. More details on this condition can be found in Rachev et al. \cite{rachev2011financial}, Kim and Lee \cite{kim2007relative}, and Choi \cite{choi2025information}.

	We next consider the CTS process, also known in the finance literature as the CGMY model \cite{carr2002fine}. The CTS process is a special case of the GTS process, obtained by imposing the following condition:
	\begin{align}
	\label{cts_condition}
		\left\{ 
		\begin{array}{ll}
			a_+=a_-=a\\
			C_+=C_-=C
		\end{array}
		\right..
	\end{align}
	This condition implies that both the upper and lower tails are controlled by the same tail index and scale parameter.

	By applying the CTS condition of Eq.~(\ref{cts_condition}) to the GTS L\'evy measure in Eq.~(\ref{levy_measure_gts}), the L\'evy measure of CTS processes is represented with
	\begin{align}
	\label{levy_measure_cts}
		\nu(dx;\boldsymbol{\xi})=C\Big(\frac{{\rm e}^{-\lambda_{+}x}}{x^{a+1}} \mathbbm{1}_{x>0}(x)+\frac{{\rm e}^{-\lambda_{-} |x|}}{|x|^{a+1}} \mathbbm{1}_{x<0}(x)\Big)dx,
	\end{align}
	where $C, \lambda_+, \lambda_-$ are positive, and $a\in(0,2)\setminus \{1\}$.
	
	It is straightforward to check that the Radon--Nikodym derivative of CTS processes is identical to that in Eq.~(\ref{radon_nikodym_gts}). As mentioned for GTS processes, the existence of the Radon--Nikodym derivative under Proposition~\ref{prp_sato} requires that all the parameters of two CTS processes, except for $\lambda_\pm$, are identical.
	
	We first consider the $\alpha$-divergences of the two models. By substituting Eqs.~(\ref{levy_measure_gts}) and (\ref{radon_nikodym_gts}) into Eq.~(\ref{a_div_measure_levy_zv}), the $\alpha$-divergence between GTS processes is obtained as
	\begin{align}
	\label{a_div_gts}
		D^{(\alpha )}(\mathbb{P}||\mathbb{Q})=\left\{ 
		\begin{array}{ll}
		\frac{4}{1-\alpha^2}\Big\{1-\exp\big[-TC_+\Gamma(-a_+)\big( \frac{1-\alpha}{2} \lambda_+^{a_+}+\frac{1+\alpha}{2} \tilde{\lambda}_+^{a_+}-( \frac{1-\alpha}{2} \lambda_++\frac{1+\alpha}{2} \tilde{\lambda}_+)^{a_+}\big)\\
		-TC_-\Gamma(-a_-)\big( \frac{1-\alpha}{2} \lambda_-^{a_-}+\frac{1+\alpha}{2} \tilde{\lambda}_-^{a_-}-( \frac{1-\alpha}{2} \lambda_-+\frac{1+\alpha}{2} \tilde{\lambda}_-)^{a_-}\big)\big]\Big\} & (\alpha \neq \pm 1)\\ 
		TC_+\Gamma(-a_+)\big((a_+-1)\lambda_+^{a_+}-a_+\tilde{\lambda}_+\lambda_+^{a_+-1}+\tilde{\lambda}_+^{a_+}\big)\\
		+TC_-\Gamma(-a_-)\big((a_--1)\lambda_-^{a_-}-a_-\tilde{\lambda}_-\lambda_-^{a_--1}+\tilde{\lambda}_-^{a_-}\big) & (\alpha =-1)\\
		TC_+\Gamma(-a_+)\big((a_+-1)\tilde{\lambda}_+^{a_+}-a_+\lambda_+\tilde{\lambda}_+^{a_+-1}+\lambda_+^{a_+}\big)\\
		+TC_-\Gamma(-a_-)\big((a_--1)\tilde{\lambda}_-^{a_-}-a_-\lambda_-\tilde{\lambda}_-^{a_--1}+\lambda_-^{a_-}\big) & (\alpha =1)
	\end{array}
	\right..
	\end{align}
	The $\alpha$-divergence between CTS processes can be derived either by applying Theorem~\ref{thm_div_levy} or by substituting the CTS condition of Eq.~(\ref{cts_condition}) into Eq.~(\ref{a_div_gts}). Using the latter approach, which is also more straightforward, the $\alpha$-divergence is obtained as
	\begin{align}
	\label{a_div_cts}
		D^{(\alpha )}(\mathbb{P}||\mathbb{Q})=\left\{ 
		\begin{array}{ll}
		\frac{4}{1-\alpha^2}\Big\{1-\exp\big[-TC\Gamma(-a)\big(\big( \frac{1-\alpha}{2} \lambda_+^{a}+\frac{1+\alpha}{2} \tilde{\lambda}_+^{a}-( \frac{1-\alpha}{2} \lambda_++\frac{1+\alpha}{2} \tilde{\lambda}_+)^{a}\big)\\
		+\big( \frac{1-\alpha}{2} \lambda_-^{a}+\frac{1+\alpha}{2} \tilde{\lambda}_-^{a}-( \frac{1-\alpha}{2} \lambda_-+\frac{1+\alpha}{2} \tilde{\lambda}_-)^{a}\big)\big)\big]\Big\} & (\alpha \neq \pm 1)\\ 
		TC\Gamma(-a)\Big(\big((a-1)\lambda_+^{a}-a\tilde{\lambda}_+\lambda_+^{a-1}+\tilde{\lambda}_+^{a}\big)
		+\big((a-1)\lambda_-^{a}-a\tilde{\lambda}_-\lambda_-^{a-1}+\tilde{\lambda}_-^{a}\big)\Big) & (\alpha =-1)\\
		TC\Gamma(-a)\Big(\big((a-1)\tilde{\lambda}_+^{a}-a\lambda_+\tilde{\lambda}_+^{a-1}+\lambda_+^{a}\big)
		+\big((a-1)\tilde{\lambda}_-^{a}-a\lambda_-\tilde{\lambda}_-^{a-1}+\lambda_-^{a}\big)\Big) & (\alpha =1)
		\end{array}
	\right..
	\end{align}
	The GTS and CTS $\alpha$-divergences in Eqs.~(\ref{a_div_gts}) and~(\ref{a_div_cts}) are identical to the corresponding results obtained in Choi \cite{choi2025information}. In particular, when $\alpha=-1$, the CTS divergence in Eq.~(\ref{a_div_cts}) coincides with the result in Kim and Lee \cite{kim2007relative}.

	Similar to the $\alpha$-divergence, the information geometry of GTS processes is found by Eqs.~(\ref{metric_levy_zv}) and (\ref{conn_levy_zv}). The Fisher information matrix of GTS processes is calculated as
	\begin{align}
	\label{metric_gts}
		g_{ij}=\Bigg(
		\begin{array}{cc}
			\frac{TC_+\Gamma(2-a_+)}{\lambda_+^{2-a_+}} &0 \\ 
			0 & \frac{TC_-\Gamma(2-a_-)}{\lambda_-^{2-a_-}}
		\end{array}\Bigg),
	\end{align}
	where $i,j$ index $\lambda_+$ and $\lambda_-$. Moreover, the $\alpha$-connection of GTS geometry is derived as
	\begin{align}
	\label{a_conn_gts_p}
		\Gamma^{(\alpha)}_{\lambda_+\lambda_+,\lambda_+}&=	-\frac{1-\alpha}{2}\frac{TC_+\Gamma(3-a_+)}{ \lambda_+^{3-a_+}},\\
	\label{a_conn_gts_m}
		\Gamma^{(\alpha)}_{\lambda_-\lambda_-,\lambda_-}&=-\frac{1-\alpha}{2}\frac{TC_-\Gamma(3-a_-)}{ \lambda_-^{3-a_-}}.
	\end{align}
	The other components vanish. Plugging the CTS condition of Eq.~(\ref{cts_condition}) into Eq.~(\ref{metric_gts}) yields the information geometry of CTS processes. The Fisher information matrix of the CTS geometry is calculated as
	\begin{align}
	\label{metric_cts}
		g_{ij}=\Bigg(
		\begin{array}{cc}
			\frac{TC\Gamma(2-a)}{\lambda_+^{2-a}} &0 \\ 
			0 & \frac{TC\Gamma(2-a)}{\lambda_-^{2-a}}
		\end{array}\Bigg),
	\end{align}
	where $i,j$ index $\lambda_+$ and $\lambda_-$. Similarly, the $\alpha$-connection of CTS process geometry is given by
	\begin{align}
	\label{a_conn_cts_p}
		\Gamma^{(\alpha)}_{\lambda_+\lambda_+,\lambda_+}&=	-\frac{1-\alpha}{2}\frac{TC\Gamma(3-a)}{ \lambda_+^{3-a}},\\
	\label{a_conn_cts_m}
		\Gamma^{(\alpha)}_{\lambda_-\lambda_-,\lambda_-}&=-\frac{1-\alpha}{2}\frac{TC\Gamma(3-a)}{ \lambda_-^{3-a}}.
	\end{align}
	All other components also vanish.	The Fisher information matrices and the $\alpha$-connections in Eqs.~(\ref{metric_gts})--(\ref{a_conn_gts_m}) and Eqs.~(\ref{metric_cts})--(\ref{a_conn_cts_m}) coincide with the corresponding GTS and CTS results derived in Choi \cite{choi2025information}.

	Due to the affine condition, easily checked with Eq.~(\ref{radon_nikodym_gts}) and $\sigma=0$, Theorem~\ref{thm_affine_levy} proves that both the GTS and CTS geometries are dually flat. Since the potential is determined up to an affine function, we choose the following potential for the GTS geometry:
	\begin{align}
	\label{potential_gts_levy}
		K_{\rm GTS}=TC_+\Gamma(-a_+)\lambda_+^{a_+}+TC_-\Gamma(-a_-)\lambda_-^{a_-}.
	\end{align}
	Its derivatives yield the dual coordinates $\eta_\pm=-TC_\pm \Gamma(1-a_\pm)\lambda_\pm^{a_\pm-1}$. The canonical and dual canonical divergences follow from Corollary~\ref{crl_affine_levy} and the $1$- and $(-1)$-divergences in Eq.~(\ref{a_div_gts}), respectively.
	Under the CTS condition, the GTS potential reduces to the following:
	\begin{align}
	\label{potential_cts_levy}
		K_{\rm CTS}=TC\Gamma(-a)(\lambda_+^a+\lambda_-^a).
	\end{align}
	The dual coordinates are given by $\eta_\pm=-TC\Gamma(1-a)\lambda_\pm^{a-1}$. The canonical and dual canonical divergences are obtained from Corollary~\ref{crl_affine_levy} and Eq.~(\ref{a_div_cts}). The potential functions, dual coordinates, and canonical divergences of the GTS and CTS geometries also coincide with the corresponding results in Choi \cite{choi2025information}.

	We can apply the information geometry of GTS and CTS processes to bias reduction and Bayesian prediction. Since Choi \cite{choi2025information} exhibited the examples for GTS processes, we briefly cover applications for CTS processes.
	
	First, the penalized log-likelihood approach \cite{firth1993bias} can be considered for the CTS process. From the Fisher information matrix in Eq.~(\ref{metric_cts}), the Jeffreys prior of CTS processes is represented with
	\begin{align}
	\label{jeffreys_cts}
		\mathcal{J}(\xi)\propto \frac{TC\Gamma(2-a)}{\sqrt{(\lambda_+\lambda_-)^{2-a}}}.
	\end{align}
	This Jeffreys prior can be used for the penalized log-likelihood in Eq.~(\ref{penalized_ll}).
	
	Moreover, we can find Komaki's Bayesian predictive priors \cite{komaki2006shrinkage} by searching potential candidates for a superharmonic function $\rho$ from the geometry. One example for a superharmonic function is as follows:
	\begin{align}
		\rho_{\pm}=\lambda_{\pm}^k,
	\end{align}
	where $0<k<a/2$, consistent with the superharmonic condition for the GTS geometry in Choi \cite{choi2025information}. 
	Another superharmonic function is the positive linear combination or multiplication of the functions above:
	\begin{align}
		\rho_1&= c_1\rho_++c_2\rho_-,\\
		\rho_2&= \rho_+\rho_-,
	\end{align}
	where $c_1>0$ and $c_2>0$. These superharmonic functions can be used in Eq.~(\ref{bayes_predict_prior}) to obtain Bayesian predictive priors.

\subsection{Variance gamma processes}
	Variance gamma (VG) processes \cite{madan1998variance} are also widely used in financial modeling. The VG process is a L\'evy process characterized by a triplet $(0, \nu, \gamma)$. Under the parameter range $a\in(0,2)\setminus\{1\}$ adopted for the tempered stable processes above, the VG process is not included in the tempered stable family considered in this paper. It is noteworthy that the L\'evy measure of the VG process can be obtained formally from the CTS L\'evy measure in the limit $a\downarrow0$.

	The L\'evy measure of VG processes is represented with
	\begin{align}
	\label{levy_measure_vg}
		\nu(dx;\boldsymbol{\xi})=C\Big(\frac{{\rm e}^{-\lambda_{+}x}}{x} \mathbbm{1}_{x>0}(x)+\frac{{\rm e}^{-\lambda_{-} |x|}}{|x|} \mathbbm{1}_{x<0}(x)\Big)dx,
	\end{align}
	where $C, \lambda_+, \lambda_-$ are positive. 
	
	The Radon--Nikodym derivative of the L\'evy measures is given by
	\begin{align}
	\label{radon_nikodym_vg}
		\frac{d\nu_{\mathbb{P}}}{d\nu_{\mathbb{Q}}}={\rm e}^{-(\lambda_+-\tilde{\lambda}_+)x}\mathbbm{1}_{x>0}(x)+{\rm e}^{-(\lambda_--\tilde{\lambda}_-)|x|}\mathbbm{1}_{x<0}(x).
	\end{align}

	We directly apply the result for L\'evy processes with $\sigma=0$. From Eq.~(\ref{delta_zv}), the positive $x$-term of $\Delta_T^{(\alpha)}(\mathbb{P}||\mathbb{Q})$ for $\alpha\neq\pm1$ is calculated as
	\begin{align}
	\label{delta_vg_p}
		\Delta_T^{(\alpha)}(\mathbb{P}||\mathbb{Q})|_{+}=TC\int_0^\infty \frac{\frac{1-\alpha}{2}{\rm e}^{-\lambda_{+}x}+\frac{1+\alpha}{2}{\rm e}^{-\tilde{\lambda}_{+}x}-{\rm e}^{-\left(\frac{1-\alpha}{2}\lambda_{+}+\frac{1+\alpha}{2}\tilde{\lambda}_{+}\right)x}}{x}\,dx.
	\end{align}
	Although the individual terms in Eq.~(\ref{delta_vg_p}) are not integrable at zero, their combination is integrable. Using the Frullani integral for $u,v>0$,
	\begin{align}
		\int_0^\infty\frac{{\rm e}^{-ux}-{\rm e}^{-vx}}{x}\,dx=\log\frac{v}{u},
	\end{align}
	we obtain the positive and negative $x$-terms of $\Delta_T^{(\alpha)}(\mathbb{P}||\mathbb{Q})$ as
	\begin{align}
	\label{delta_vg_closed}
		\Delta_T^{(\alpha)}(\mathbb{P}||\mathbb{Q})|_{\pm}
		=TC\Big(-\frac{1-\alpha}{2}\log\lambda_\pm-\frac{1+\alpha}{2}\log\tilde{\lambda}_\pm+\log\big(\frac{1-\alpha}{2}\lambda_\pm+\frac{1+\alpha}{2}\tilde{\lambda}_\pm\big)\Big).
	\end{align}
	These expressions are finite when
	\begin{align}
		\frac{1-\alpha}{2}\lambda_\pm+\frac{1+\alpha}{2}\tilde{\lambda}_\pm>0.
	\end{align}
	Since $\lambda_\pm$ and $\tilde{\lambda}_\pm$ are positive, this condition is automatically satisfied for $|\alpha|<1$.

	For $\alpha=-1$, the positive $x$-term can also be calculated directly from Eq.~(\ref{a_div_measure_levy_zv}) as
	\begin{align}
	\begin{split}
		D^{(-1)}(\mathbb{P}||\mathbb{Q})|_{+}&=TC\int_0^\infty\big[-(\lambda_+-\tilde{\lambda}_+){\rm e}^{-\lambda_+x}+\frac{{\rm e}^{-\tilde{\lambda}_+x}-{\rm e}^{-\lambda_+x}}{x}\big]dx\\
		&=TC\Big(\frac{\tilde{\lambda}_+}{\lambda_+}-1+\log\lambda_+-\log\tilde{\lambda}_+\Big).
	\end{split}
	\end{align}
	The negative $x$-term is obtained in the same way. The case $\alpha=1$ follows from the $\alpha$-duality. By substituting Eq.~(\ref{delta_vg_closed}) into Eq.~(\ref{a_div_in_delta}), we obtain the $\alpha$-divergence between two VG processes as
	\begin{align}
	\label{a_div_vg}
		D^{(\alpha )}(\mathbb{P}||\mathbb{Q})=\left\{
		\begin{array}{ll}
		\frac{4}{1-\alpha^2}\Big\{1-\exp\big[-TC\big(\big( -\frac{1-\alpha}{2} \log{\lambda_+}-\frac{1+\alpha}{2} \log{\tilde{\lambda}_+}+\log{( \frac{1-\alpha}{2} \lambda_++\frac{1+\alpha}{2} \tilde{\lambda}_+)\big)}\\
		+\big( -\frac{1-\alpha}{2}\log \lambda_--\frac{1+\alpha}{2}\log \tilde{\lambda}_- + \log( \frac{1-\alpha}{2} \lambda_-+\frac{1+\alpha}{2} \tilde{\lambda}_-)\big)\big)\big]\Big\} & (\alpha \neq \pm 1)\\
		TC\big((\tilde{\lambda}_+/\lambda_+-1+\log{\lambda_+}-\log{\tilde{\lambda}_+})+(\tilde{\lambda}_-/\lambda_--1+\log{\lambda_-}-\log{\tilde{\lambda}_-})\big) & (\alpha =-1)\\
		TC\big((\lambda_+/\tilde{\lambda}_+-1+\log{\tilde{\lambda}_+}-\log{\lambda_+})+(\lambda_-/\tilde{\lambda}_--1+\log{\tilde{\lambda}_-}-\log{\lambda_-})\big) & (\alpha =1)
		\end{array}
		\right..
	\end{align}

	From the $\alpha$-divergence of Eq.~(\ref{a_div_vg}), the Fisher information matrix of VG processes is represented with
	\begin{align}
	\label{metric_vg}
		g_{ij}=\Bigg(
		\begin{array}{cc}
			\frac{TC}{\lambda_+^{2}} &0 \\
			0 & \frac{TC}{\lambda_-^{2}}
		\end{array}\Bigg),
	\end{align}
	where $i,j$ index $\lambda_{\pm}$.

	Similarly, the $\alpha$-connection of VG process geometry is given by
	\begin{align}
	\label{a_conn_vg_p}
		\Gamma^{(\alpha)}_{\lambda_+\lambda_+,\lambda_+}&=	-(1-\alpha)\frac{TC}{ \lambda_+^{3}},\\
	\label{a_conn_vg_m}
		\Gamma^{(\alpha)}_{\lambda_-\lambda_-,\lambda_-}&=-(1-\alpha)\frac{TC}{ \lambda_-^{3}}.
	\end{align}
	All other components vanish. It is easy to check that the VG geometry is e-flat. Eq.~(\ref{radon_nikodym_vg}) also shows that the logarithm of the Radon--Nikodym derivative between the L\'evy measures is affine in $(\lambda_+,\lambda_-)$. By applying Theorem~\ref{thm_affine_levy}, the VG geometry is dually flat. For the VG geometry, consider the potential:
	\begin{align}
	\label{potential_vg_levy}
		K_{\rm VG}=-TC(\log\lambda_++\log\lambda_-).
	\end{align}
	Its derivatives yield the dual coordinates $\eta_\pm=-TC/\lambda_\pm$. The canonical and dual canonical divergences follow from Corollary~\ref{crl_affine_levy} and Eq.~(\ref{a_div_vg}).

	Similar to the GTS and CTS processes, the information geometry of VG processes can be applied to bias reduction. Based on Eq.~(\ref{metric_vg}), the Jeffreys prior is represented with
	\begin{align}
	\label{jeffreys_vg}
		\mathcal{J}_{\rm VG}(\boldsymbol{\xi})\propto \frac{TC}{\lambda_+\lambda_-}.
	\end{align}
	For the exponential-family parametrization of the positive and negative gamma components, the penalized log-likelihood in Eq.~(\ref{penalized_ll}) gives the Firth adjustment. The adjusted score equations are obtained as
	\begin{align}
	\label{firth_score_vg}
		\frac{\partial l}{\partial \lambda_\pm}-\frac{1}{\lambda_\pm}=0.
	\end{align}

	We next consider $n$ independent full-path observations. Let $S_+$ and $S_-$ denote the sums of the positive and negative jump magnitudes, respectively, over the $n$ sample paths. Since each one-sided component is a gamma process, the log-likelihood for $\lambda_\pm$ is represented as
	\begin{align}
	\label{vg_component_likelihood}
		l(\lambda_\pm)=nTC\log\lambda_\pm-\lambda_\pm S_\pm+\mathrm{constant}.
	\end{align}
	Using Eq.~(\ref{firth_score_vg}), the maximum likelihood and Firth estimators are obtained as
	\begin{align}
	\label{vg_mle_firth_closed}
		\widehat{\lambda}_{\pm,\mathrm{MLE}}=\frac{nTC}{S_\pm},\qquad
		\widehat{\lambda}_{\pm,\mathrm{Firth}}=\frac{nTC-1}{S_\pm},
	\end{align}
	where the Firth estimator has an interior solution when $nTC>1$.

	For a numerical example, we test the estimators with $C=0.1$, $T=1$, and the true parameters $(\lambda_+,\lambda_-)=(1,1.5)$. For $n=50,100,200,500$, we repeat the simulation $5,000$ times. For each parameter, $S_\pm$ is generated from the gamma distribution with shape $nTC$ and rate $\lambda_\pm$, corresponding to the full-path sufficient statistic in Eq.~(\ref{vg_component_likelihood}). Figure~\ref{fig_vg_firth} presents the relative bias and RMSE. The maximum likelihood estimators have positive relative biases for small $n$. The relative biases of the Firth estimators remain close to zero. Additionally, the Firth estimators have lower RMSEs than the maximum likelihood estimators for both $\lambda_+$ and $\lambda_-$ at all sample sizes. The differences become smaller as $n$ increases.

	\begin{figure}[!htb]
	\centering
	\begin{minipage}{0.48\textwidth}
	\centering
	\includegraphics[width=\textwidth]{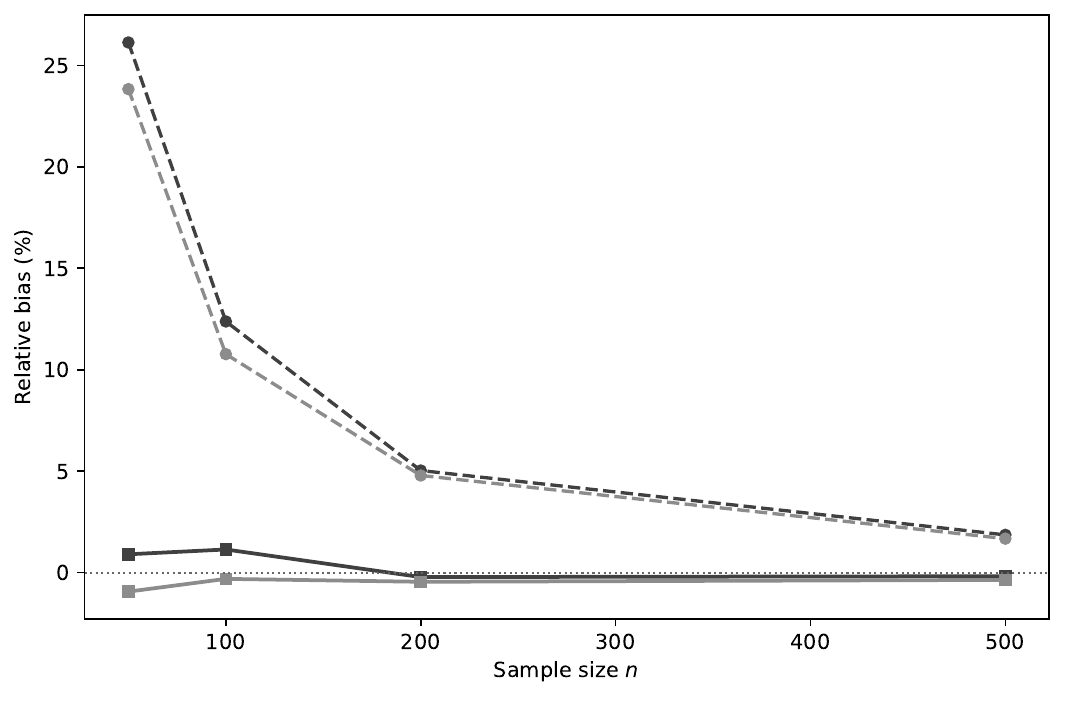}\\
	{\small (a) Relative bias}
	\end{minipage}\hfill
	\begin{minipage}{0.48\textwidth}
	\centering
	\includegraphics[width=\textwidth]{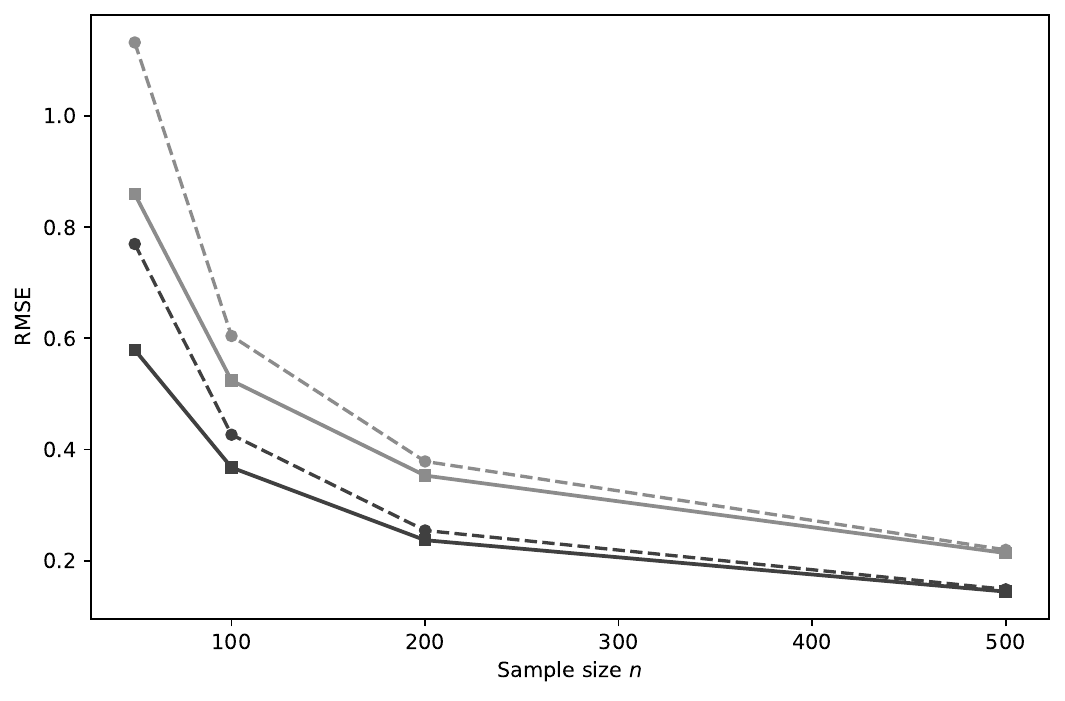}\\
	{\small (b) RMSE}
	\end{minipage}
	\caption{Relative biases and RMSEs of the maximum likelihood and Firth estimators for $\lambda_+$ and $\lambda_-$ in the VG model. Solid-square and dashed-circle lines denote the Firth and maximum likelihood estimators, respectively. Dark and light gray lines correspond to $\lambda_+$ and $\lambda_-$, respectively. The results are based on $5,000$ Monte Carlo replications with $C=0.1$, $T=1$, and $(\lambda_+,\lambda_-)=(1,1.5)$.}
	\label{fig_vg_firth}
	\end{figure}

	Finally, it is noteworthy that the VG divergence and the metric tensor obtained above are also consistent with the formal limit $a\downarrow0$ of the corresponding CTS results. However, the same limiting argument does not provide a non-trivial Bayesian predictive prior. For the CTS geometry, the power function $\rho_{\pm}=\lambda_{\pm}^k$ is superharmonic when $0<k<a/2$. As $a\downarrow0$, this interval collapses and there is no admissible non-zero $k$. This is also consistent with the VG geometry itself. Let $x=\log\lambda_+$ and $y=\log\lambda_-$. From Eq.~(\ref{metric_vg}), the metric tensor is represented as
	\begin{align}
		ds^2=TC(dx^2+dy^2),
	\end{align}
	and the parameter space corresponds to $\mathbb{R}^2$. Since a positive superharmonic function on the whole Euclidean plane is constant, the power-function construction used for CTS processes does not provide a non-trivial global Bayesian predictive prior for VG processes.

\subsection{Merton model}
	The Merton model \cite{merton1976option} extends the Black--Scholes model by introducing a Poisson jump component in addition to geometric Brownian motion. The jump component allows the model to describe discontinuous movements in asset prices and has been used for option pricing.
	
	Unlike the tempered stable processes and VG processes studied earlier, the Merton model is a L\'evy process with L\'evy triplet $(\sigma, \nu, \gamma)$, i.e., $\sigma\neq0$. Its L\'evy measure is given by
	\begin{align}
		\label{levy_measure_merton}
		\nu(dx)=\frac{\lambda}{\delta\sqrt{2\pi}}\exp{\Big(-\frac{(x-m)^2}{2\delta^2}\Big)}dx,
	\end{align}
	where $\lambda,\delta>0$ and $m\in \mathbb{R}$.
	
	Let us consider two risk-neutral Merton models with the same volatility of $\sigma$. For these risk-neutral Merton models, 
the $\alpha$-divergence between two Merton models can be derived from Eq.~(\ref{a_div_measure_levy_martingale}) in Theorem~\ref{thm_div_levy}:
	\begin{align}
	 \label{a_div_measure_merton_martingale}
		D^{(\alpha )}(\mathbb{P}||\mathbb{Q})=\left\{ 
		\begin{array}{ll}
		\frac{4}{1-\alpha ^{2}}\Big(1-{\rm{e}}^{-\Delta^{(\alpha)}_T(\mathbb{P}||\mathbb{Q})}\Big) & (\alpha \neq \pm 1)\\ 
		\frac{T}{2\sigma^2}\big(\lambda(\exp{(m+\frac{\delta^2}{2})}-1)-\tilde{\lambda}(\exp{(\tilde{m}+\frac{\tilde{\delta}^2}{2})}-1)\big)^2+ & \\
		T\big((\tilde{\lambda}-\lambda)+\lambda(\ln{(\frac{\lambda}{\delta})}-\ln{(\frac{\tilde{\lambda}}{\tilde{\delta}})}-\frac{1}{2}+\frac{(m-\tilde{m})^2+\delta^2}{2\tilde{\delta}^2})\big) & (\alpha =-1)\\
		\frac{T}{2\sigma^2}\big(\lambda(\exp{(m+\frac{\delta^2}{2})}-1)-\tilde{\lambda}(\exp{(\tilde{m}+\frac{\tilde{\delta}^2}{2})}-1)\big)^2+ & \\
		T\big((\lambda-\tilde{\lambda})+\tilde{\lambda}(\ln{(\frac{\tilde{\lambda}}{\tilde{\delta}})}-\ln{(\frac{\lambda}{\delta})}-\frac{1}{2}+\frac{(m-\tilde{m})^2+\tilde{\delta}^2}{2\delta^2})\big) & (\alpha =1)
		\end{array}
	\right.,
	\end{align}
	where $\Delta_T^{(\alpha)}(\mathbb{P}||\mathbb{Q})$ is given by
	\begin{align}
		\Delta_T^{(\alpha)}(\mathbb{P}||\mathbb{Q})&=\frac{1-\alpha^2}{4}\frac{T}{2\sigma^2}\big(\lambda(\exp{(m+\frac{\delta^2}{2})}-1)-\tilde{\lambda}(\exp{(\tilde{m}+\frac{\tilde{\delta}^2}{2})}-1)\big)^2\nonumber\\
		&\quad + T\Big(\frac{1-\alpha}{2}\lambda+\frac{1+\alpha}{2}\tilde{\lambda}-\Big(\frac{\lambda}{\delta}\Big)^{\frac{1-\alpha}{2}}\Big(\frac{\tilde{\lambda}}{\tilde{\delta}}\Big)^{\frac{1+\alpha}{2}}(\frac{1-\alpha}{2}\frac{1}{\delta^2}+\frac{1+\alpha}{2}\frac{1}{\tilde{\delta}^2})^{-1/2}\nonumber\\
		&\exp{\big[\frac{1}{2}(\frac{1-\alpha}{2}\frac{1}{\delta^2}+\frac{1+\alpha}{2}\frac{1}{\tilde{\delta}^2})^{-1}(\frac{1-\alpha}{2}\frac{m}{\delta^2}+\frac{1+\alpha}{2}\frac{\tilde{m}}{\tilde{\delta}^2})^2-\frac{1}{2}(\frac{1-\alpha}{2}\frac{m^2}{\delta^2}+\frac{1+\alpha}{2}\frac{\tilde{m}^2}{\tilde{\delta}^2})\big]}\Big).\nonumber
	\end{align}
	For $\alpha\neq\pm1$, the expression for $\Delta_T^{(\alpha)}(\mathbb{P}||\mathbb{Q})$ is finite when
	\begin{align}
		\frac{1-\alpha}{2\delta^2}+\frac{1+\alpha}{2\tilde{\delta}^2}>0.
	\end{align}
	Since $\delta$ and $\tilde{\delta}$ are positive, this condition is automatically satisfied for $|\alpha|<1$.
	It is straightforward to check that the Kullback--Leibler divergence in Eq.~(\ref{a_div_measure_merton_martingale}) is identical to that in Cont and Tankov \cite{cont2004nonparametric}.
	
	From Eq.~(\ref{metric_levy_martingale}) in Theorem~\ref{thm_geo_levy} or Eq.~(\ref{metric_levy_delta}) in Corollary~\ref{crl_geo_levy_delta}, the Fisher information matrix of the Merton model is given by
	\begin{align}
		\label{metric_merton}
		g_{ij} = T\begin{pmatrix} 
		\frac{1}{\lambda}+\frac{(e^{m+\delta^2/2}-1)^2}{\sigma^2} & 
		\frac{\lambda(e^{m+\delta^2/2}-1)e^{m+\delta^2/2}}{\sigma^2} & 
		\frac{\lambda\delta(e^{m+\delta^2/2}-1)e^{m+\delta^2/2}}{\sigma^2} \\
		\frac{\lambda(e^{m+\delta^2/2}-1)e^{m+\delta^2/2}}{\sigma^2} & 
		\frac{\lambda}{\delta^2}+\frac{\lambda^2 e^{2(m+\delta^2/2)}}{\sigma^2} & 
		\frac{\lambda^2\delta e^{2(m+\delta^2/2)}}{\sigma^2} \\
		\frac{\lambda\delta(e^{m+\delta^2/2}-1)e^{m+\delta^2/2}}{\sigma^2} & 
		\frac{\lambda^2\delta e^{2(m+\delta^2/2)}}{\sigma^2} & 
		\frac{2\lambda}{\delta^2}+\frac{\lambda^2\delta^2 e^{2(m+\delta^2/2)}}{\sigma^2}
		\end{pmatrix},
	\end{align}
	where the coordinate system is $(\lambda, m, \delta)$.
	
	Similarly, the $\alpha$-connection for information geometry of the Merton model is derived from Eq.~(\ref{conn_levy_martingale}) in Theorem~\ref{thm_geo_levy} or Eq.~(\ref{conn_levy_delta}) in Corollary~\ref{crl_geo_levy_delta}:
	\begin{align}
		\label{conn_merton}
		\begin{aligned}
		&\Gamma^{(\alpha)}_{\lambda\lambda,\lambda}=-\frac{T(1+\alpha)}{2\lambda^2},\textrm{ }
		\Gamma^{(\alpha)}_{\lambda m,\lambda}=T\frac{e^{m+\frac{\delta^2}{2}}\left(e^{m+\frac{\delta^2}{2}}-1\right)}{\sigma^2},\textrm{ }
		\Gamma^{(\alpha)}_{\lambda\delta,\lambda}=T\frac{\delta e^{m+\frac{\delta^2}{2}}\left(e^{m+\frac{\delta^2}{2}}-1\right)}{\sigma^2},\\
		&\Gamma^{(\alpha)}_{mm,\lambda}	=T\left(\frac{\lambda\left(e^{2m+\delta^2}-e^{m+\frac{\delta^2}{2}}\right)}{\sigma^2}-\frac{1+\alpha}{2\delta^2}\right),\textrm{ }
		\Gamma^{(\alpha)}_{m\delta,\lambda}=	T\frac{\lambda\delta e^{m+\frac{\delta^2}{2}}\left(e^{m+\frac{\delta^2}{2}}-1\right)}{\sigma^2},\textrm{ }\\
		&\Gamma^{(\alpha)}_{\delta\delta,\lambda}=	T\left(\frac{\lambda(\delta^2+1)\left(e^{2m+\delta^2}-e^{m+\frac{\delta^2}{2}}\right)}{\sigma^2}-\frac{1+\alpha}{\delta^2}\right),\\
		&\Gamma^{(\alpha)}_{\lambda\lambda,m}=0,\textrm{ }
		\Gamma^{(\alpha)}_{\lambda m,m}=T\left(\frac{\lambda e^{2m+\delta^2}}{\sigma^2}+\frac{1-\alpha}{2\delta^2}\right),\textrm{ }
		\Gamma^{(\alpha)}_{\lambda\delta,m}=T\frac{\lambda\delta e^{2m+\delta^2}}{\sigma^2},\\
		&\Gamma^{(\alpha)}_{mm,m}=T\frac{\lambda^2 e^{2m+\delta^2}}{\sigma^2},\textrm{ }
		\Gamma^{(\alpha)}_{m\delta,m}=T\left(\frac{\lambda^2\delta e^{2m+\delta^2}}{\sigma^2}-\frac{(1+\alpha)\lambda}{\delta^3}\right),\textrm{ }
		\Gamma^{(\alpha)}_{\delta\delta,m}=T\frac{\lambda^2(\delta^2+1)e^{2m+\delta^2}}{\sigma^2},\\
		&\Gamma^{(\alpha)}_{\lambda\lambda,\delta}=0,\textrm{ }
		\Gamma^{(\alpha)}_{\lambda m,\delta}=T\frac{\lambda\delta e^{2m+\delta^2}}{\sigma^2},\textrm{ }
		\Gamma^{(\alpha)}_{\lambda\delta,\delta}=	T\left(\frac{\lambda\delta^2 e^{2m+\delta^2}}{\sigma^2}+\frac{1-\alpha}{\delta^2}	\right),\\
		&\Gamma^{(\alpha)}_{mm,\delta}=T\left(\frac{\lambda^2\delta e^{2m+\delta^2}}{\sigma^2}+\frac{(1-\alpha)\lambda}{\delta^3}\right),\textrm{ }
		\Gamma^{(\alpha)}_{m\delta,\delta}=T\frac{\lambda^2\delta^2 e^{2m+\delta^2}}{\sigma^2},\\
		&\Gamma^{(\alpha)}_{\delta\delta,\delta}=T\left(\frac{\lambda^2\delta(\delta^2+1)e^{2m+\delta^2}}{\sigma^2}-\frac{2\lambda(1+2\alpha)}{\delta^3}\right).
		\end{aligned}
	\end{align}

	Unlike the GTS, CTS, and VG cases, $c_*(\boldsymbol{\xi})$ and $\psi_*(x;\boldsymbol{\xi})$ for the Merton model are not affine in $(\lambda,m,\delta)$. Since the affine condition in Theorem~\ref{thm_affine_levy} is not satisfied, the potential representation and dual flatness in Theorem~\ref{thm_affine_levy} are not directly obtained for the Merton model.
	
\section{Conclusion}
\label{sec_con}
	In this paper, we derived the $\alpha$-divergence between L\'evy processes and obtained the corresponding Fisher information matrix and $\alpha$-connection. We also obtained a potential function generating the Fisher information matrix and the $\alpha$-connection under the affine condition. The resulting geometry is dually flat, and its dual coordinates are obtained from the gradient of the potential. Moreover, the $1$- and $(-1)$-divergences coincide with the canonical and dual canonical divergences, respectively.

	The general results were applied to several L\'evy processes. For GTS, CTS, and VG processes, the logarithmic Radon--Nikodym derivatives between the L\'evy measures satisfy the affine condition. We obtained their potential functions and dual coordinates explicitly, and the results for GTS and CTS processes coincide with those previously obtained for tempered stable processes. We also applied the general formulas to the Merton model with a nonzero diffusion component and derived its $\alpha$-divergence, Fisher information matrix, and $\alpha$-connection.

	Additionally, we considered statistical applications of the derived geometries. For CTS processes, the Fisher information matrix provides the Jeffreys prior and candidates for Bayesian predictive priors. For VG processes, we derived the Firth-adjusted estimators for the tempering parameters and compared them with the maximum likelihood estimators through numerical experiments. The numerical results show that the Firth estimators have smaller relative biases and RMSEs, particularly for small sample sizes.
	
\section*{Acknowledgment}
	We thank Hyangju Kim and Young Shin Kim for useful discussions on L\'evy processes.
	
\bibliographystyle{plain}
\bibliography{ig_levy}

\begin{thebibliography}{10}

\bibitem{amari2000methods}
Shun-ichi Amari and Hiroshi Nagaoka.
\newblock {\em Methods of information geometry}, volume 191.
\newblock American Mathematical Soc., 2000.

\bibitem{anand2017equity}
Abhinav Anand, Tiantian Li, Tetsuo Kurosaki, and Young~Shin Kim.
\newblock {The equity risk posed by the too-big-to-fail banks: a Foster--Hart estimation}.
\newblock {\em Annals of Operations Research}, 253:21--41, 2017.

\bibitem{barbaresco2006information}
Fr{\'e}d{\'e}ric Barbaresco.
\newblock Information intrinsic geometric flows.
\newblock In {\em AIP Conference Proceedings}, volume 872, pages 211--218. American Institute of Physics, 2006.

\bibitem{barbaresco2012information}
Fr{\'e}d{\'e}ric Barbaresco.
\newblock {Information geometry of covariance matrix: Cartan-Siegel homogeneous bounded domains, Mostow/Berger fibration and Frechet median}.
\newblock In {\em Matrix information geometry}, pages 199--255. Springer, 2012.

\bibitem{beck2013empirical}
Alexander Beck, Young Shin~Aaron Kim, Svetlozar Rachev, Michael Feindt, and Frank Fabozzi.
\newblock {Empirical analysis of ARMA-GARCH models in market risk estimation on high-frequency US data}.
\newblock {\em Studies in Nonlinear Dynamics and Econometrics}, 17(2):167--177, 2013.

\bibitem{carr2002fine}
Peter Carr, H{\'e}lyette Geman, Dilip~B Madan, and Marc Yor.
\newblock {The fine structure of asset returns: An empirical investigation}.
\newblock {\em The Journal of Business}, 75(2):305--332, 2002.

\bibitem{choi2021kahlerian}
Jaehyung Choi.
\newblock {K{\"a}hler information manifolds of signal processing filters in weighted Hardy spaces}.
\newblock {\em arXiv preprint arXiv:2108.07746}, 2021.

\bibitem{choi2025information}
Jaehyung Choi.
\newblock Information geometry of tempered stable processes.
\newblock {\em arXiv preprint arXiv:2502.12037}, 2025.

\bibitem{choi2024diversified}
Jaehyung Choi, Hyangju Kim, and Young~Shin Kim.
\newblock Diversified reward-risk parity in portfolio construction.
\newblock {\em Studies in Nonlinear Dynamics \& Econometrics}, 2024.

\bibitem{choi2015reward}
Jaehyung Choi, Young~Shin Kim, and Ivan Mitov.
\newblock Reward-risk momentum strategies using classical tempered stable distribution.
\newblock {\em Journal of Banking \& Finance}, 58:194--213, 2015.

\bibitem{choi2015geometric}
Jaehyung Choi and Andrew~P Mullhaupt.
\newblock {Geometric shrinkage priors for K{\"a}hlerian signal filters}.
\newblock {\em Entropy}, 17(3):1347--1357, 2015.

\bibitem{choi2015kahlerian}
Jaehyung Choi and Andrew~P Mullhaupt.
\newblock K{\"a}hlerian information geometry for signal processing.
\newblock {\em Entropy}, 17(4):1581--1605, 2015.

\bibitem{cont2004nonparametric}
Rama Cont and Peter Tankov.
\newblock Nonparametric calibration of jump-diffusion option pricing models.
\newblock {\em The Journal of Computational Finance}, 7:1--49, 2004.

\bibitem{efron1975defining}
Bradley Efron.
\newblock Defining the curvature of a statistical problem (with applications to second order efficiency).
\newblock {\em The Annals of Statistics}, pages 1189--1242, 1975.

\bibitem{efron1978geometry}
Bradley Efron.
\newblock The geometry of exponential families.
\newblock {\em The Annals of Statistics}, pages 362--376, 1978.

\bibitem{firth1993bias}
David Firth.
\newblock Bias reduction of maximum likelihood estimates.
\newblock {\em Biometrika}, 80(1):27--38, 1993.

\bibitem{georgiev2015periodic}
Krastyu Georgiev, Young~Shin Kim, and Stoyan Stoyanov.
\newblock Periodic portfolio revision with transaction costs.
\newblock {\em Mathematical Methods of Operations Research}, 81(3):337--359, 2015.

\bibitem{jacod2003limit}
Jean Jacod and Albert~N. Shiryaev.
\newblock {\em Limit Theorems for Stochastic Processes}, volume 288 of {\em Grundlehren der mathematischen Wissenschaften}.
\newblock Springer, Berlin, Heidelberg, 2 edition, 2003.

\bibitem{kass2011geometrical}
Robert~E Kass and Paul~W Vos.
\newblock {\em Geometrical foundations of asymptotic inference}.
\newblock John Wiley \& Sons, 2011.

\bibitem{kim2023deep}
Young~Shin Kim, Hyangju Kim, and Jaehyung Choi.
\newblock {Deep Calibration With Artificial Neural Network: A Performance Comparison on Option Pricing Models}.
\newblock {\em The Journal of Financial Data Science}, 5(4):100--118, 2023.

\bibitem{kim2007relative}
Young~Shin Kim and Jeong~Hyun Lee.
\newblock {The relative entropy in CGMY processes and its applications to finance}.
\newblock {\em Mathematical Methods of Operations Research}, 66:327--338, 2007.

\bibitem{kim2009computing}
Young~Shin Kim, Svetlozar Rachev, Michele~Leonardo Bianchi, and Frank~J Fabozzi.
\newblock {Computing VaR and AVaR in infinitely divisible distributions}.
\newblock 2009.

\bibitem{kim2010tempered}
Young~Shin Kim, Svetlozar~T Rachev, Michele~Leonardo Bianchi, and Frank~J Fabozzi.
\newblock {Tempered stable and tempered infinitely divisible GARCH models}.
\newblock {\em Journal of Banking \& Finance}, 34(9):2096--2109, 2010.

\bibitem{kim2011time}
Young~Shin Kim, Svetlozar~T Rachev, Michele~Leonardo Bianchi, Ivan Mitov, and Frank~J Fabozzi.
\newblock Time series analysis for financial market meltdowns.
\newblock {\em Journal of Banking \& Finance}, 35(8):1879--1891, 2011.

\bibitem{komaki2006shrinkage}
Fumiyasu Komaki.
\newblock {Shrinkage priors for Bayesian prediction}.
\newblock {\em The Annals of Statistics}, 34(2):808--819, 2006.

\bibitem{kosmidis2009bias}
Ioannis Kosmidis and David Firth.
\newblock Bias reduction in exponential family nonlinear models.
\newblock {\em Biometrika}, 96(4):793--804, 2009.

\bibitem{kosmidis2010generic}
Ioannis Kosmidis and David Firth.
\newblock A generic algorithm for reducing bias in parametric estimation.
\newblock {\em Electronic Journal of Statistics}, 4:1097--1112, 2010.

\bibitem{kuchler1997exponential}
Uwe K{\"u}chler and Michael S{\o}rensen.
\newblock {\em Exponential Families of Stochastic Processes}.
\newblock Springer Series in Statistics. Springer, New York, 1997.

\bibitem{madan1998variance}
Dilip~B Madan, Peter~P Carr, and Eric~C Chang.
\newblock The variance gamma process and option pricing.
\newblock {\em Review of Finance}, 2(1):79--105, 1998.

\bibitem{merton1976option}
Robert~C Merton.
\newblock Option pricing when underlying stock returns are discontinuous.
\newblock {\em Journal of Financial Economics}, 3(1-2):125--144, 1976.

\bibitem{oda2021shrinkage}
Hidemasa Oda and Fumiyasu Komaki.
\newblock Shrinkage priors on complex-valued circular-symmetric autoregressive processes.
\newblock {\em IEEE Transactions on Information Theory}, 67(8):5318--5333, 2021.

\bibitem{rachev2011financial}
Svetlozar~T Rachev, Young~Shin Kim, Michele~L Bianchi, and Frank~J Fabozzi.
\newblock {\em Financial models with L{\'e}vy processes and volatility clustering}.
\newblock John Wiley \& Sons, 2011.

\bibitem{ravishanker2001differential}
Nalini Ravishanker.
\newblock {Differential geometry of ARFIMA processes}.
\newblock {\em Communications in Statistics-Theory and Methods}, 30(8-9):1889--1902, 2001.

\bibitem{ravishanker1990differential}
Nalini Ravishanker, Edward~L Melnick, and Chih-Ling Tsai.
\newblock {Differential geometry of ARMA models}.
\newblock {\em Journal of Time Series Analysis}, 11(3):259--274, 1990.

\bibitem{sato1999levy}
Ken-Iti Sato.
\newblock {\em L{\'e}vy processes and infinitely divisible distributions}, volume~68.
\newblock Cambridge University Press, 1999.

\bibitem{tanaka2018superharmonic}
Fuyuhiko Tanaka.
\newblock Superharmonic priors for autoregressive models.
\newblock {\em Information Geometry}, 1(2):215--235, 2018.

\bibitem{tanaka2008superharmonic}
Fuyuhiko Tanaka and Fumiyasu Komaki.
\newblock A superharmonic prior for the autoregressive process of the second-order.
\newblock {\em Journal of Time Series Analysis}, 29(3):444--452, 2008.

\bibitem{tsuchida2012mean}
Naoshi Tsuchida, Xiaoping Zhou, and Svetlozar Rachev.
\newblock {Mean-ETL portfolio selection under maximum weight and turnover constraints based on fundamental security factors}.
\newblock {\em Journal of Investing}, 21(1):14, 2012.

\end{thebibliography}

\end{document}